\documentclass[12pt]{amsart}
\usepackage{geometry}

\usepackage[mathscr]{eucal}
\usepackage{amsmath,amssymb,amscd,amsthm}
\usepackage{color}
\usepackage{epsfig}
\newtheorem{theorem}{Theorem}

\newtheorem{lemma}{Lemma}

\newtheorem{proposition}{Proposition}
\theoremstyle{definition}
\newtheorem{remark}{Remark}

\theoremstyle{plane}

\def \beq{ \begin{equation} }
\def \eeq{\end{equation}}

\usepackage{graphicx}
\usepackage{amssymb}
\usepackage{epstopdf}

\def \S {\mathbb {S}}

\def \cO{\mathcal{O}}
\def \cR{\mathcal{R}}
\def \eps {\varepsilon}
\def \ii {\mbox i}
\def \fm {\phantom{-}}

\title{On the stability of tetrahedral relative equilibria in the positively curved 4-body problem}
\begin{document}
\maketitle
\markboth{Florin Diacu, Regina Mart\'{\i}nez, Ernesto P\'erez-Chavela, and Carles Sim\'o}{Stability in the positively curved $4$-body problem}
\author{\begin{center}{\bf Florin Diacu}\\
Pacific Institute for the Mathematical Sciences\\ and\\
Department of Mathematics and Statistics\\
University of Victoria\\
Victoria, Canada\\
\bigskip
{\bf Regina Mart\'{\i}nez}\\
Departament de Matem\`atiques\\
Universitat Aut\`onoma de Barcelona\\
Bellatera, Barcelona, Spain\\
\bigskip
{\bf Ernesto P\'erez-Chavela}\\
Departamento de Matem\'aticas\\
Universidad Aut\'onoma Metropolitana\\
Iztapalapa, Mexico, D.F., Mexico\\
 \bigskip
{\bf Carles Sim\'o}\\
Departament de Matem\`atica Aplicada i An\`alisi\\
Universitat de Barcelona\\
Barcelona, Spain
\end{center}}

\medskip

\begin{abstract} We consider the motion of point masses given by a natural extension of Newtonian gravitation to spaces of constant positive curvature. Our goal is to explore the spectral stability of tetrahedral orbits of the corresponding 4-body problem in the 2-dimensional case, a situation that can be reduced to studying the motion of the bodies on the unit sphere. We first perform some extensive and highly precise numerical experiments to find the likely regions of stability and instability, relative to the values of the masses and to the latitude of
the position of three equal masses. Then we support the numerical evidence 
with rigorous analytic proofs in the vicinity of some limit cases  in which 
certain masses are either very large or negligible, or the latitude is close to zero.
\end{abstract}

\hfill

\section{Introduction}

The goal of this paper is to study the spectral stability of tetrahedral orbits in the 2-dimensional positively curved 4-body problem, i.e.\ when four point particles of positive masses move on the unit sphere $\mathbb S^2$ according to a gravitational law that naturally extends the Newtonian potential to spaces of constant curvature. This is a particular case of the curved $n$-body problem, $n\ge 2$, which has been only recently derived in a suitable setting, both for constant positive curvature (i.e.\ 2- and 3-dimensional spheres) and for constant negative curvature (i.e.\ 2- and 3-dimensional hyperbolic spheres), \cite{diacupersan1}, \cite{diacupersan2}, \cite{diacu4}, \cite{diacu5}.

The case $n\!=\!2$ dates back to the 1830s, when J\'anos Bolyai and Nikolai Lobachevsky introduced it for the hyperbolic space $\mathbb H^3$, \cite{bol}, \cite{lob}. The equations of motion for constant nonzero curvature are given by the cotangent of the spherical distance, in the positive case, and the hyperbolic cotangent of the hyperbolic distance when the space curvature is negative. For zero curvature, the classical Newtonian equations of the $n$-body problem are recovered. The analytic expression of the potential is due to Ernest Schering, \cite{sch1}, \cite{sch2}, for negative curvature, and to Wilhelm Killing, \cite{kil1}, \cite{kil2}, \cite{kil3}, for positive curvature. The problem also became established due to the results of Heinrich Liebmann, \cite{lie1}, \cite{lie2}, \cite{lie3}. The attempts to extend the problem to spaces of variable curvature started with Tullio Levi-Civita, \cite{civ1}, \cite{civ2}, Albert Einstein, Leopold Infeld, Banesh Hoffman, \cite{ein}, and Vladimir Fock, \cite{foc}, and led to the equations of the post-Newtonian approximation, which are useful in many applications, including the global positioning system. But unlike in the case of constant curvature, these equations are too large and complicated to allow an analytic approach.

It is important to ask why the above extension of the Newtonian potential to spaces of constant curvature is natural, since there is no unique way of generalizing the classical equations of motion in order to recover them when the space in which the bodies move flattens out. The reason is that the cotangent potential is, so far, the only one known to satisfy the same basic properties as the Newtonian potential in its simplest possible setting, that of one body moving around a fixed centre, the so-called Kepler problem \cite{koz}. Two basic properties stick out in this case: the potential of the classical Kepler problem is a harmonic function in $\mathbb R^3$, i.e.\ it satisfies Laplace's equation, and it generates a central field in which all bounded orbits are closed, a result proved by Joseph Louis Bertrand in 1873, \cite{ber}.

On one hand, the cotangent potential approaches the classical Newtonian potential when the curvature tends to zero, whether through positive or negative values. On the other hand, the cotangent potential satisfies Bertrand's property for the curved Kepler problem and is a solution of the Laplace-Beltrami equation, \cite{diacu5}, \cite{koz}, the natural generalization of Laplace's equation to Riemannian and pseudo-Riemannian manifolds, which include the spaces of constant positive curvature $\kappa> 0$ we are interested in here.

In the Euclidean case, the Kepler problem and the 2-body problem are equivalent. The reason for this overlap is the existence of the linear momentum and centre of mass integrals. It can be shown with their help that the equations of motion are identical, whether the origin of the coordinate system is fixed at the centre of mass or fixed at one of the two bodies. For nonzero curvature, however, things change. The equations of motion of the curved $n$-body problem lack the linear momentum and centre of mass integrals, which prove to characterize only the Euclidean case, \cite{diacu5}, \cite{diacu6}, \cite{diacupersan1}. Consequently the curved Kepler problem and the curved 2-body problem are not equivalent anymore. It turns out that, as in the Euclidean case, the curved Kepler problem is Liouville integrable, but, unlike in the Euclidean case, the curved 2-body problem is not, \cite{shc2}, \cite{shc3}, \cite{shc4}. As expected, the curved $n$-body problem is not integrable for $n\ge 3$, a property also known to be true in the Euclidean case.

A detailed bibliography and a history of these developments appear in \cite{diacu5}. Notice also that the study we perform here in $\mathbb S^2$ is not restrictive since the qualitative behaviour of the orbits is independent of the value of the positive curvature, \cite{diacu5}, \cite{diacupersan1}.

The current paper is a natural continuation of some ideas developed in \cite{masimo}, which studied the stability of Lagrangian orbits (rotating equilateral triangles) of the curved 3-body problem on the unit sphere, $\mathbb S^2$, both when the mutual distances remain constant and when they vary in time.
The former orbits, called relative equilibria, are a particular case of the latter, and they are part of the backbone towards understanding the equations of motion in the dynamics of particle systems, \cite{diacu5}, \cite{diacuper}. Unlike in the classical Newtonian 3-body problem, where the motion of Lagrangian orbits takes place in the Euclidean plane, the Lagrangian orbits of $\mathbb S^2$ exist only when the three masses are equal, \cite{diacupersan1}, \cite{diacu5}. But equal-mass classical Lagrangian orbits are known to be unstable, so it was quite a surprise to discover that, in $\mathbb S^2$, the Lagrangian relative equilibria exhibit two zones of linear stability. This does not seem to be the case for constant negative curvature, i.e.\ in the hyperbolic plane $\mathbb H^2$, as some preliminary numerical experiments show. Consequently, the shape of the physical space has a strong influence over particle dynamics, therefore studies in this direction promise to lead to new connections between the classical and the curved $n$-body problem.

The result obtained in \cite{masimo} thus opened the door to investigations into the stability of other orbits characteristic to $\mathbb S^2$, and tetrahedral solutions came as a first natural choice, since the experience accumulated in the previous study could be used in this direction, as we will actually do here.

The paper is organized as follows. In Section 2, we introduce the tetrahedral solutions in $\mathbb S^2$, i.e.\ orbits of the $4$-body problem with one body of mass $m_1$ fixed at the north pole and the other three bodies of equal mass $m$ located at the vertices of a
rotating equilateral triangle orthogonal to the $z$-axis. If the triangle is above the equator, i.e.\ the $z$ coordinate of the three equal masses is positive, the tetrahedral relative equilibria exist for any given masses. If the triangle is below the equator, the relative equilibria exist just for some values of the masses. To approach the spectral stability of the relative equilibria, we compute in Section 3 the Jacobian matrix of the vector field at the relative equilibria, which become fixed points in the rotating frame. 

The study of the stability starts in Section 4, where we analyze three limit problems, first taking the mass at the north pole $m_1=0$, i.e.\ $\Gamma:=m_1/m=0$. In this case, the tetrahedral relative equilibria are spectrally stable for $z<0$ and unstable for $z>0$. In the second limit problem the mass $m_1$ is very large when compared to the other three masses. Taking $\eps:=1/\Gamma=m/m_1$, in the limit case $\eps\to 0$, the problem reduces to three copies of $2$-body problems, formed for the mass at the north pole and a body of zero mass. The changes in this degenerate situation for small $\eps>0$ are studied in Section 6, where we also consider the third limit problem, for which we take $z=0$ and let the parameters $\Gamma$ or $\eps$ and $z$ move away from zero. To reach this point, we previously perform in Section 5 a deep and highly precise numerical analysis to determine the regions of stability according to the values of $z$ and of the masses. Our main results occur in Section 6, where using the Newton polygon (including the degenerate cases) and the Implicit Function Theorem we study all bifurcations that appear when the limit problems are perturbed and draw rigorously proved conclusions about the spectral stability of tetrahedral relative equilibria. We end this paper with a full bifurcation diagram and an outline of future research perspectives.

\section{Tetrahedral orbits in $\S^2$}

Consider four bodies of masses $m_1,m_2,m_3,m_4>0$ moving on the unit sphere
$\mathbb S^2$, which has constant curvature 1. Then the natural extension of Newton's equations of motion from $\mathbb R^2$ to $\mathbb S^2$ is given by
\begin{equation}
\label{negativee}
\ddot{\bf q}_i=\sum_{j=1,j\ne i}^4\frac{m_j[{\bf q}_j-({\bf q}_i\cdot {\bf q}_j){\bf q}_i]}{[1-({\bf q}_i\cdot {\bf q}_j)^2]^{3/2}} - (\dot{\bf q}_i\cdot \dot{\bf q}_i){\bf q}_i,\,\
({\bf q}_i\cdot{\bf q}_i)=1, \ \ i=1,2,3,4,
\end{equation}
where the vector ${\bf q}_i=(x_i,y_i,z_i)$ gives the position of the body of mass $m_i, i=1,2,3,4$, and the dot, $\cdot$ , denotes the standard scalar product of $\mathbb R^3$, \cite{diacupersan2}, \cite{perezrey}. These equations are known to be Hamiltonian, \cite{diacu5}.

By a tetrahedral solution we mean an orbit in which one body, say $m_1$, is fixed at the north pole $(0,0,1)$, while the other bodies, $m_2=m_3=m_4=:m$, lie at the vertices of an equilateral triangle that rotates uniformly in a plane parallel with the equator $z=0$. In other words, we are interested in solutions of the form
\begin{align*}
x_1&=0,& y_1&=0,& z_1&=1,\\
x_2&=r\cos\omega t,& y_2&=r\sin\omega t,& z_2&=\pm(1-r^2)^{1/2},\\
x_3&=r\cos(\omega t+2\pi/3),& y_3&=r\sin(\omega t+2\pi/3),& z_3&=\pm(1-r^2)^{1/2},\\
x_4&=r\cos(\omega t+4\pi/3),& y_4&=r\sin(\omega t+4\pi/3),& z_4&=\pm(1-r^2)^{1/2},
\end{align*}
where $r$ and $\omega$ are constant, $r$ denotes the radius of the circle
in which the triangle rotates, and $\omega$ represents the angular velocity of the
rotation. A straightforward computation shows that
$$
\omega^2=\frac{24m}{r^3(12-9r^2)^{3/2}}\pm\frac{m_1}{r^3(1-r^2)^{1/2}}=:g(r),
$$
where we take the plus or the minus sign depending on whether $z:=z_2=z_3=z_4$ is positive or negative, respectively.

The purpose of this paper is to provide a complete study of the spectral stability of such orbits, which are obviously periodic. For this, we will use rotating coordinates, in which the above periodic relative equilibria become fixed points for the equations of motion. Recall that a fixed point is linearly stable if all orbits of the tangent flow are bounded for all time, and it is spectrally stable if no eigenvalue is positive or has positive real part. Linear stability implies spectral stability, but not the other way around. 
Nevertheless, spectral stability fails to imply linear stability only in the case of
matrices with multiple eigenvalues whose associated Jordan block is not
diagonal.

To achieve our goal, we further consider the coordinate and time-rescaling transformations
$$
{\bf q}_i=(x_i,y_i,z_i)\to{\bf Q}_i=(X_i,Y_i),\ \ t=r^{3/2}\tau,
$$
$$x_i=rX_i,\ \ y_i=rY_i, \ \  z_i=\pm [1-r^2(X_i^2+Y_i^2)]^{1/2}, \ i=1,2,3,4. $$
A simple computation shows that if we choose $\omega t=\Omega\tau,$
the angular velocity relative to the new time variable $\tau$ takes the form
$$
\Omega=\pm\bigg[{\frac{24m}{(12-9r^2)^{3/2}}\pm\frac{m_1}{(1-r^2)^{1/2}}\bigg]^{1/2}}.
$$
With the above transformations, and using the fact that
$$X'_i=r^{1/2}\dot{x}_i,\ Y'_i=r^{1/2}\dot{y}_i, \ X''_i=r^2\ddot{x}_i, \ Y''_i=r^2\ddot{y}_i, \ i=1,2,3,4,  $$
the equations of motion become
$$
{\bf Q}_i''= r^3 \sum_{j=1,j\ne i}^4\frac{m_j({\bf Q}_j-f_{ij}{\bf Q}_i)}{(1-f_{ij}^2)^{3/2}}-r^3(\dot{{\bf q}}_i\cdot\dot{{\bf q}}_i){\bf Q}_i,\ \ i=1,2,3,4,
$$
where $'\!=\!\frac{d}{d\tau}$ and $f_{ij}\!=\!({\bf q}_i \cdot{\bf q}_j)\!=\!
r^2(X_iX_j+Y_iY_j)\!+\! z_i z_j$. 

We further introduce the rotating coordinates $\xi_i, \eta_i, \ i=1,2,3,4,$ with
$$
\begin{pmatrix} X_i\\ Y_i \end{pmatrix} =
{\mathcal R}(\Omega\tau)
\begin{pmatrix} \xi_i\\ \eta_i \end{pmatrix}, \ i=1,2,3,4,\  {\rm where}\
{\mathcal R}(\Omega\tau)=
\begin{pmatrix}
\cos\Omega\tau & -\sin\Omega\tau\\
\sin\Omega\tau & \fm\cos\Omega\tau
\end{pmatrix}.
$$
Then $\xi_i\xi_j+\eta_i\eta_j=X_iX_j+Y_iY_j, \ i,j\in\{1,2,3,4\}$, expressions that take the value 1 when $i=j$. Moreover,
$$
\begin{pmatrix} \xi''_i\\ \eta''_i \end{pmatrix}
=\Omega^2 \begin{pmatrix} \xi_i\\ \eta_i \end{pmatrix}
+2\Omega \begin{pmatrix} \fm\eta'_i\\ -\xi'_i \end{pmatrix}
+{\mathcal R}^{-1}(\Omega\tau) \begin{pmatrix} X''_i\\ Y''_i \end{pmatrix},
\ i=1,2,3,4.
$$
A straightforward computation shows that the new equations of motion have the form
\begin{eqnarray}
\label{rotcoord}
\left(\!\begin{array}{c} \xi_i^{\prime\prime} \\ \eta_i^{\prime\prime} \\\end{array}\!\right)
\!\!\!\!&=&\!\!\!\!
2 \Omega \!\left(\!\begin{array}{c} \fm\eta_i^\prime \\ -\xi_i^\prime \\\end{array}\!\right)
\!+\!\Omega^2\!\left(\!\begin{array}{c} \xi_i \\ \eta_i \\\end{array}\!\right)\! -\!
r^2 h_i\!\left(\!\begin{array}{c} \xi_i \\ \eta_i \\\end{array}\!\right)\!+
\!\!\sum_{\substack{j=1\\ j\ne i}}^4\!\! m_j g_{i,j}^{-\frac{3}{2}} \!\left[\!
\left(\!\begin{array}{c} \xi_j \\ \eta_j \\\end{array}\!\right)\! -\! f_{i,j}
\left(\!\begin{array}{c} \xi_i \\ \eta_i \\\end{array}\!\right)\! \right],
\end{eqnarray}
where
\begin{equation}
\label{Omega}
\Omega^2={\frac{24m}{(12-9r^2)^{3/2}}\pm\frac{m_1}{(1-r^2)^{1/2}}},
\end{equation}
$$
p_{i,j} = \xi_i \xi_j+ \eta_i \eta_j, \quad \rho_i^2=\xi_i^2+\eta_i^2 \quad
z_{i,j} = (1-r^2 \rho_i^2)(1-r^2 \rho_j^2),
$$
$$g_{i,j}= \rho_i^2 + \rho_j^2 - 2 s_{i,j} p_{i,j} \sqrt{z_{i,j}}
- r^2 (p_{i,j}^2 + \rho_i^2 \rho_j^2),
$$
$$
h_i =  \Omega^2 \rho_i^2 +
2 \Omega (\xi_i \eta_i^\prime - \eta_i \xi_i^\prime)
+ ((\xi_i^\prime)^2 + (\eta_i^\prime)^2) + \frac{r^2}{1- r^2\rho_i^2}
(\xi_i \xi_i^\prime + \eta_i \eta_i^\prime)^2,
$$
$$
f_{i,j} = r^2 (\xi_i \xi_j+\eta_i\eta_j) + z_i z_j
=r^2 (\xi_i \xi_j+\eta_i\eta_j) + s_{i,j} \sqrt{z_{i,j}},
$$
$$
s_{i,j} = {\rm sign}(z_i z_j) = \left\{
\begin{array}{cc} {\rm sign}(z):=s, & \quad i=1 \;\; {\rm or} \;\; j=1, \\
1, & \quad i, j \neq 1, \\ \end{array} \right.
$$
which implies that $z=s\sqrt{1-r^2}$.

Before we start to study the stability of the tetrahedral relative equilibria, we must see for
what values of the masses they exist. For this purpose,
we will prove the following result.

\begin{proposition}
Consider a tetrahedral orbit of the curved $4$-body problem in $\mathbb S^2$
with the mass $m_1>0$ fixed at the north pole $(0,0,1)$ and the masses
$m_2=m_3=m_4=:m>0$ fixed at the vertices of an equilateral triangle that
rotates uniformly on $\mathbb S^2$ in a plane parallel with the equator $z=0$. Then, if
the triangle is above the equator, i.e.\ $0<z<1$, tetrahedral relative equilibria exist for any given masses.
If the triangle is below the equator, i.e.\ $-1<z<0$, then

(i) if $0<m_1<\frac{16m}{9\sqrt{3}}$, for any positive value of $\Omega^2$ up to
a maximum it can attain, tetrahedral relative equilibria exist.
In this case, if  $0<m_1 \leq \frac{m}{\sqrt{3}}$,
for any positive value of $\Omega^2$, such that
$0 < \Omega^2< \frac{1}{\sqrt{3}} - \frac{m_1}{m}$,  
there is a unique $z$ with $-1< z < 0$ 
corresponding to a relative equilibrium. Otherwise, there are 
 two distinct values of $z\in(-1,0)$, each
corresponding to a different relative equilibrium;

(ii) if $m_1>\frac{16m}{9\sqrt{3}}$, then there are no tetrahedral relative equilibria.
\end{proposition}

\begin{proof}
From equation \eqref{Omega}, a tetrahedral relative equilibrium must satisfy the condition
$$
\Omega^2 = m \left(\frac{\Gamma}{z} + \frac{8 }{\sqrt{3}(1+3 z^2)^{3/2}} \right),
\ \ {\rm where}\ \ z= \pm \sqrt{1-r^2}, \quad \Gamma = \frac{m_1}{m}.
$$

\begin{figure}[htbp]
\vspace*{-4mm}
\begin{tabular}{c}
\hspace*{-6mm}\epsfig{file=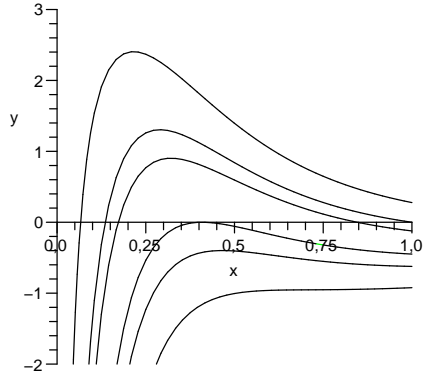,width=190pt} \\
\end{tabular}
\vspace*{-6mm}
\caption{Plot of $F(u,\Gamma) $, $u \in [0,1]$,  
for six distinct values of $\Gamma$. From top to bottom, we took
$\Gamma=0.3,\, 1/\sqrt{3},\,  0.7, \, 16/(9\sqrt{3})$,  $1.2$ and $1.5$, respectively.}
\label{fig1}
\end{figure}

For $0<z<1$, since the right hand side is positive, the statement in the proposition
is obvious. For $-1<z<0$, we introduce $u=-z$. Then the 
equation above can be written as
\begin{eqnarray}
\label{equation3}
\Omega^2 = m F(u;\Gamma),
\end{eqnarray}
where
$$
F(u;\Gamma) = -\frac{\Gamma}{u} + \frac{8 }{\sqrt{3}(1+3 u^2)^{3/2}}.
$$
The behaviour of $F(u;\Gamma) $ for different values of $ \Gamma $ is summarized 
in the Figure \ref{fig1}.
The critical points of $ F(u;\Gamma) $ satisfy $ \Gamma = g(u) $
where
$$
g(u) = \frac{72 u^3}{\sqrt{3} (1+3 u^2)^{5/2}}.
$$
The function $ g(u)$ has a maximum at $u=1/\sqrt{2} $ equal to $ 144/(25 \sqrt{15})$.
For any $ 0 < \Gamma < 144/(25 \sqrt{15}) $ there exist values 
$ 0 < u_1 < 1/\sqrt{2} < u_2 $ such that $ \frac{d F}{d u} =0$. 
It is clear that $ F $ has a maximum at $ u_1 $ and a minimum at $u_2$.
Moreover, 
$$
F(u_1,\Gamma) = \frac{\Gamma(1-6 u_1^2)}{9 u_1^3}
$$
and
$$
\frac{d F}{d u} (1/\sqrt{6}; \Gamma) = 6 \left(\Gamma - \frac{16}{9 \sqrt{3}}\right).
$$
If $ \Gamma > \frac{16}{9 \sqrt{3}} $ then $ u_1 > 1/\sqrt{6} $. Therefore 
$ F(u_1,\Gamma) <0 $ and $ F(u,\Gamma) <0 $ for any $ u \in (0,1)$. Therefore equation (\ref{equation3}) has no real solutions and there are no tetrahedral relative equilibria.
If $\Gamma < \frac{16}{9 \sqrt{3}} $, then $ F(u_1,\Gamma) >0 $ and for any value 
of $ \Omega^2 $ smaller than $ m F(u_1; \Gamma) $  equation (\ref{equation3})
has real solutions. In this case, if $ 0 < \Gamma < \frac{1}{\sqrt{3}} $ and $ \Omega^2 < F(1; \Gamma) = 1/\sqrt{3} -\Gamma $, there is a unique real solution $ u \in (0,1) $, otherwise there are exactly two real solutions $ 0 < u_{R1} < u_{R2} < 1$ that correspond to relative equilibria. 

So, if $ 1/\sqrt{3} < \Gamma < 16/(9 \sqrt{3}) $, we can say that for any positive value of $ \Omega^2$
such that
$$
0 < \frac{\Omega^2}{m} < \frac{\Omega_a^2}{m}, \quad
\Omega_a^2 :=  \max \{F(r,\Gamma) | r \in (0,1) \},
$$
there are two values of $r$,
$\, 0 < r_{E1}< r_{E2} < 1 \,$ with $\Omega^2 = m F(r,\Gamma)$.
If $\Gamma>16/(9 \sqrt{3})$, there are no tetrahedral relative equilibria.
This remark completes the proof.
\end{proof}

\section{The characteristic polynomial}\label{sect:equ}

The goal of this section is to obtain the characteristic polynomial, which will allow us to compute the spectrum of the Jacobian matrix corresponding to a tetrahedral 
relative equilibrium. The computations we perform and the conclusions we draw here will prepare the terrain for understanding the stability of the orbit relative to $\Gamma$ and $z$.

In equations \eqref{rotcoord}, which describe the motion in rotating coordinates, the tetrahedral relative equilibrium becomes the fixed point
$$
\left(\!\begin{array}{c} \xi_1 \\ \eta_1 \\\end{array}\right) =
\left(\!\begin{array}{c} 0\\ 0\\ \end{array}\!\right), \!\quad
\left(\!\begin{array}{c} \xi_2 \\ \eta_2 \\\end{array}\!\right) =
\left(\!\begin{array}{c} 1\\ 0\\ \end{array}\!\right), \!\quad
\left(\!\begin{array}{c} \xi_3 \\ \eta_3 \\\end{array}\!\right) =
\left(\!\begin{array}{c} -1/2 \\ \sqrt{3}/2  \\ \end{array}\!\right), \!\quad
\left(\!\begin{array}{c} \xi_4 \\ \eta_4 \\\end{array}\!\right) =
\left(\!\begin{array}{c} -1/2 \\ -\sqrt{3}/2  \\ \end{array}\!\right),
$$
$$
\xi_i^\prime = \eta_i^\prime =0, \quad i=1,2,3,4.
$$
It is now convenient to introduce the linear operators $S_d$ and $S_o$ acting on $2\times 2$ matrices. $S_d$ changes the signs of the elements on the diagonal, whereas $S_o$ changes the signs of the other remaining elements.

Long but straightforward computations show that the Jacobian matrix corresponding to the vector field $f$ corresponding to system (2) at the fixed point is given by the matrix
$$Df=\left( \begin{array}{cc} 0 & I \\ A & B \\ \end{array} \right),
$$
where
$$B = \Omega\,{\rm diag} (B_1,B_2,B_3,B_4),\ \ A=F+C, \ \ F=\Omega^2\,{\rm diag} (F_1,F_2,F_3,F_4),$$
$$
B_1\!=\!\left(\!\!\!\begin{array}{rc} 0 & 2\\ -2 & \!0 \end{array}\!\!\!\right),\;\;
B_2\!=\! 2\!\left(\!\!\!\begin{array}{rl} 0 & \!z^2\\ -1 & \!0 \end{array}\!\!\!\right),\;\;
B_3\!=\! \frac{1}{2}\! \left(\!\!\!\begin{array}{cc} \sqrt{3}(z^2\!-\!1) & \!3\!-\!z^2 \\
-1\!-\!3 z^2 & \!\sqrt{3} (1\!-\!z^2)\end{array}\!\!\!\right),\;\; B_4\!=\!S_d(B_3), $$
\vspace{2mm}
$$ F_1\! =\! I, \;\; F_2\! =\! \left(\!\begin{array}{cc} -1\!+\! 2 z^2 & \!0 \\
0 & \!0 \\ \end{array}\!\right),\;\;
F_3\! =\! \frac{1}{4}
\left(\!\begin{array}{cc} -1\!+\! 4z^2& \sqrt{3}(1\!-\!4 z^2) \\ \sqrt{3}(1\!-\!4 z^2)& -3\!+\!12z^2 \\
\end{array}\!\right), \;\; F_4\! =\! S_o(F_3), $$
\vspace{2mm}
$$ C= \left(\begin{array}{rrrr} mC_{11} & mC_{12} & mC_{13} & mC_{14} \\
m_1\, C_{21} & X C_{22} & X C_{23} & X C_{24} \\
m_1\, C_{31} & X C_{32} & X C_{33} & X C_{34} \\
m_1\, C_{41} & X C_{42} & X C_{43} & X C_{44} \\ \end{array}\right),
\quad  X=m G^{-5/2}, \quad G=\frac{3}{4} (1+3 z^2),$$
\vspace{2mm}
$$
C_{12}\!=\!\left(\begin{array}{cc} - 2 & 0 \\ 0 & 1 \\ \end{array}\right), \quad
C_{13}\! =\! \frac{1}{4}\! \left(\begin{array}{cc} 1 & 3 \sqrt{3} \\
3 \sqrt{3} & - 5  \end{array}\right),
\quad C_{14}\!=\!S_o(C_{13}),\quad
$$
$$
C_{21}\! =\! \left(\!\!\begin{array}{cc} -2 z^2 &\! 0 \\ 0 & \!1 \\
\end{array}\!\!\right), \quad
C_{31}\!\!=\!\!\frac{1}{4}\!\left(\!\!\!\begin{array}{cc}
3- 2 z^2 &\!\!\!\!\!\!\sqrt{3}(1\!+\!2z^2)\\
\sqrt{3}(1\!+\!2z^2)&\!1 - 6z^2\end{array}\!\!\!\right)\!, \;
\quad C_{41}\! =\!S_o(C_{31}),
$$
\vspace{2mm}
$$
C_{23}\! =\! \frac{3}{8}  \left(\!\!\begin{array}{cc}
-1+9 z^2 -18 z^4  & \sqrt{3} (1-z^2 +6 z^4) \\
3 \sqrt{3} (-1+3 z^2) & 5-3 z^2 \end{array}\!\!\right),
\quad C_{24}\! =\!S_o(C_{23}),
$$
\vspace{2mm}
$$
C_{32}\!\! =\!\! \frac{3}{8}  \!\left(\!\!\!\begin{array}{cc}
-9 z^4\! -\! 6z^2\!+\!5 & \!\!\sqrt{3} (3 z^4\!+\!4 z^2\!-\!1) \\
3 \sqrt{3} (3 z^4\!-\!2 z^2\!+\!1) &  -9z^4+12 z^2 -1 \end{array}\!\!\!\right)\!,
$$
\vspace{2mm}
$$
C_{34}\!\! =\!\! \frac{1}{4}  \!\left(\!\!\!\begin{array}{cc}
3\!+\!9 z^2 & 3 \sqrt{3} (3 z^4- 5 z^2 + 2) \\
0 & \!\!\!\!\!3 - 27 z^4 \end{array}\!\!\!\right)\!,\;
C_{42}\! =\!S_o(C_{32}),\; C_{43}\! =\!S_o(C_{34}),
$$
\vspace{2mm}
$$
C_{11}=\frac{3z}{2}\;I,\quad
C_{22}=\frac{3}{4}\; \left(\!\begin{array}{cc} 1-15 z^2 & 0\\ 0& -2 +12 z^2 \end{array}\!\right),
$$
$$
C_{33}=\frac{3}{16}\; \left(\!\begin{array}{cc} -5+ 21 z^2 & 3\sqrt{3}(1-9 z^2) \\ 
3\sqrt{3}(1-9 z^2) & 1-33 z^2 
\end{array}\!\right),\quad C_{44}=S_o(C_{33}).
$$
So, we can write
$$ \Omega^2 = \frac{m_1}{z} + 3 m G^{-3/2}.$$

The eigenvalues of $Df$ are the zeroes of the polynomial
$$ \det(-\zeta^2 I + \zeta B +A) =0.$$
Let us define $\mu$ such that $ \zeta=\Omega \mu $. Then the characteristic
equation becomes
$$ p(\mu) = \det (-\Omega^2 \mu^2 I + \mu \Omega B + A) =0.$$
Let us introduce
$$ S:= - \Omega^2 \mu^2 I + \mu \Omega B + F+C = P(\mu) + C, \qquad
P(\mu):=- \Omega^2 \mu^2 I + \mu \Omega B+ F.  $$
Then
$$ P(\mu)= \Omega^2{\rm diag}(P_1(\mu),P_2(\mu),P_3(\mu),P_4(\mu)), $$
where
$$
P_1(\mu) =  \left(\begin{array}{cc}
1- \mu^2 & 2 \mu \\ - 2 \mu & 1- \mu^2 \\ \end{array} \right), \qquad
P_2(\mu)\! =\! \left(\begin{array}{cc}
-1 - \mu^2 +4 z^2 & 2 \mu z^2 \\ -2 \mu & -\mu^2 \\ \end{array} \right),
$$
$$
P_3(\mu) = \left(\begin{array}{cc}
- \mu^2 - \frac{\sqrt{3}}{2} \mu(1-z^2)  - \frac{1}{4} + z^2 &
\frac{\mu}{2} ( 3 + z^2) + \sqrt{3}(\frac{1}{4}- z^2) \\
-\frac{\mu}{2} (1 + 3 z^2) + \sqrt{3}(\frac{1}{4}- z^2) &
- \mu^2 + \frac{\sqrt{3}}{2}\mu (1-z^2) - \frac{3}{4} + 3 z^2 \\\end{array} \right),
$$
and  $P_4(\mu)$ follows from $P_3(\mu)$ by changing the
sign in $\sqrt{3}$. Then $ p(\mu)= \det(S)$.

Before computing $p(\mu)$, it is convenient to perform some reduction and introduce additional notations. The first integrals associated to the energy and the $SO(2)$ invariance give rise in $p(\mu)$ to the factors $\mu^2$ and $\mu^2+1$, which we can ignore.  To get further, recall first that, if the upper index $^T$ denotes the transposed of a matrix, a $2n\times 2n$ matrix $A$ is called infinitesimal symplectic if it satisfies the equation
$$
JA+A^TJ=0,\ \ {\rm where}\ \ J=\begin{pmatrix}0& I_n\cr 
-I_n& 0\cr \end{pmatrix}\ \ {\rm and} \ \ I_n\ \ {\rm is \ the \ unit\ matrix}.
$$
As the matrix $P(\mu)$ is infinitesimal symplectic, $p(\mu)$ contains only even powers of $\mu$ and, hence, we obtain with the notation $M=:\mu^2$ a simpler expression. We can further reduce the problem by considering a unique mass parameter. In general, we can discuss the stability in terms of the mass ratio $\Gamma=m_1/m$, thus skipping the dependence on $m$. However, to study some limit cases, it will be also useful to  consider $ \varepsilon=m/m_1 $ instead of $ \Gamma $. From now on we will
simply denote the previous $p(\mu)$ by $\hat{p}(M)$, after changing the variable and 
skipping the factors $M$ and $M+1$.

The characteristic polynomial $\hat{p}(M)$ has degree 6 in $M$,
and its coefficients are polynomials of degree 8 in $\Gamma$ 
that depend on $z$. The dependence on $z$ is not of polynomial type due to the factors $G^{-5/2}$ and $\Omega^2$. 
An important difference relative to the curved 3-body problem is that these factors cannot
be easily ``canceled'' when multiplying by a power of $\Omega^2$, unless we take
$m_1=0$. Introducing
$$ 
D=D(z)=G^{-5/2} = \alpha(1+3z^2)^{-5/2}, \ \ {\rm with}\ \ \alpha=\frac{32}{9\sqrt{3}}, 
$$
the expression of $\hat{p}(M)$ becomes a huge polynomial, which can be fortunately simplified in part.

Indeed, the factor $F=4 z \Omega^2 $  appears in $\hat{p}(M)$ with multiplicity 3. 
Skipping it and further renaming the quotient as $\hat{p}(M)$, we obtain a
polynomial of degree 6 in $M$ whose coefficients have degrees 21 in $z$ and 5
in $D(z)$ and $\Gamma$. It is clear that the dependence on $D$ can be decreased
to degree 1, but then the degree in $z$ increases. No other obvious factors
appear. Whenever necessary, we will make the dependence on the other variables explicit
by writing $\hat{p}(M,\Gamma,z,D(z))$.

As it is usually done in the 3-body problem, we can look for values of $z$ and $\Gamma$ related to bifurcations of the zeroes of $\hat{p}(M)$ that lead to changes in the spectrum: either $M=0$ is a root or $\hat{p}(M)$ has a negative root
with multiplicity at least equal to two. In the former case, after dividing by the factor $z^2$, the polynomial $\hat{p}(0,\Gamma,z,D(z))$ has degrees 19, 5, and 3 relative to $z,D(z)$, and $\Gamma$, respectively. In the latter case, after dividing by the factor $F^{27}z^{25}D(z)^2$, the resultant of $\hat{p}(M)$ and $\frac{d}{dM} \hat{p}(M)$ produces a polynomial, denoted by Res$(\Gamma,z,D(z))$, that has degrees 104, 25, and 13 in $z,D(z)$, and $\Gamma$, respectively. (Recall that if two polynomials $P$ and $Q$
have the roots $a_1, a_2,\dots, a_\nu$ and $b_1,b_2,\dots, b_\eta$, respectively,
then they have a common root if and only if Res$(P,Q)=0$, where
$
{\rm Res}(P,Q):=\prod_{i=1}^\nu\prod_{j=1}^\eta(a_i-b_j)
$
is their resultant. In the present case $M$ has to be seen as the variable of
the polynomials and $\Gamma$ and $z$ as parameters.)
Certainly, it can happen that $\hat{p}(M), \hat{p}(0,\Gamma,z,D(z))$, or Res$(\Gamma,z,D(z))$ have some other non-trivial factor. But the dependence in $D$ makes hard to recognize it.

Hence, to study the stability problem, we will combine a numerical scan of the changes in the solutions $M_i,i=1,\ldots,6$, for some grids in $\Gamma$ and $z$, 
with the theoretical analysis done in the vicinity of some limit problems, which we will next introduce.

\section{Three limit problems} \label{sect:limit}

Before proceeding with our numerical computations it is worth studying the behaviour of the system in some simple limit cases, which we will later use to achieve our main goal of understanding the spectral stability of tetrahedral orbits.

\subsection{\bf The restricted problem}  \label{sect:limitm10}

If we take $ m_1=0$, which is equivalent with $\Gamma=0$, the matrix $S$ has the block structure
$$ S = \left( \begin{array}{cc} \Omega^2 P_{1}(\mu)+ m C_{11} & \tilde{C} \\
0 & Z(\mu) \\ \end{array}\right), $$
where
$$
\Omega^2 P_1(\mu)+ m C_{11} = \left( \begin{array}{cc} \Omega^2 (1-\mu^2) + \frac{3 m}{2} z &
2 \mu \Omega^2 \\
- 2 \mu \Omega^2 & \Omega^2 (1-\mu^2) + \frac{3 m}{2} z \\
\end{array} \right) $$
and  $Z(\mu)$ is a $6\times 6$ matrix such that all the terms have either a factor
$ \Omega^2 = 3 m G^{-3/2}$ or a factor $ X= mG^{-5/2}$.
Then
$$
\det(S)= \det (\Omega^2 P_{1}(\mu)+ m C_{11}) \det(Z(\mu)).
$$
Note that from the matrix $Z(\mu)$ we recover the eigenvalues, and so the spectral stability of the Lagrangian orbits of the curved 3-body problem studied in \cite{masimo}. 
These results, to be used in the next section, can be summarized as follows.
The determinant of $ Z(\mu) $ is a polynomial in $M $.
After eliminating the  factors $ M $, $ M^2 +1 $, and the exact solution given by $ M_0 = -2 z^2(5-3 z^2)/(1+3 z^2) $, we obtain a polynomial of degree 3 in $M$, $Q(M)$ (see also \cite{masimo}), with polynomial coefficients in $r^2 = 1 -z^2$.
In \cite{masimo} it was proved that there exist three values of $r$, $ 0 < r_1< r_2< r_3 < 1$, where  Hamiltonian-Hopf bifurcations occur, such that, for $ r \in (r_1,r_2) \cup (r_3,1)$, the zeroes of $ Q $ are negative, and consequently those Lagrangian orbits for the curved 3-body problem are linearly (and not only
spectrally) stable. For $ r \in (0,r_1) \cup (r_2,r_3) $, $Q$ has a pair of complex zeroes, so the corresponding Lagrangian orbits are unstable.
(For more details about Hamiltonian-Hopf bifurcations see \cite{meer}.)

In the restricted case, the stability of the zero-mass body located at $(0,0,1)$ can be obtained by studying the matrix 
$\Omega^2 P_1(\mu)+ mC_{11}$. 
A simple computation shows that
$$ \frac{1}{\Omega^4} \det (\Omega^2 P_1(\mu)+ m C_{11})=\mu^4-c \mu^2 + \Big(\frac{c}{2}+2\Big)^2,
\qquad c= z G^{3/2} -2. $$
Then
\begin{eqnarray}
M=\mu^2=\frac{1}{2}(c\pm\sqrt{-8(c+2)})=\frac{1}{2}(c\pm\sqrt{-8zG^{3/2}}).\label{zeropm}
\end{eqnarray}
If $z >0$, $ \mu^2$ becomes a complex number with real part different from zero.
But if $z<0$, we obtain a couple of negative values for $M$, with an only exception that appears for $c=-4$, i.e.\ $z^2(1+3z^2)^3=256/27$ ($z\approx
-0.73176195875$). For this $z$, one of the values of $M$ is zero and the other
value is negative. Of course, when $z$ moves away from this exceptional value,
the zero value of $M$ becomes negative again. So, we can draw the following conclusion.

\begin{proposition} \label{prop:rest}
Considering the dynamics of the infinitesimal mass in the above restricted problem, the tetrahedral relative equilibrium is spectrally stable for negative values of $z$, but unstable for positive $z$.
\end{proposition}

\subsection{\bf The $1 + 3$ limit case: $m=0$ } \label{subs:1+3}

As opposed to the previous restricted case, we now study the problem in which the
body lying at $(0,0,1)$ is massive, whereas the other three bodies have zero mass. It is easy to see that
$$ \det(S) = (\Omega^2)^8 \det (P_1(\mu)) \det (P_2(\mu))\det (P_3(\mu))\det (P_4(\mu)) =
(\Omega^2)^8 \mu^6 (1+\mu^2)^5. $$
Skipping the trivial factors, the characteristic equation $\hat{p} (M)=0 $ reduces in the limit to
\[ T(M)=M^2(M+1)^4=0, \]
which yields the roots $\mu=0, \mu=+\ii$, and $\mu=-\ii$, all of them of multiplicity
4. In other words, all the characteristic multipliers are equal to 1.

This outcome is not unexpected. Indeed, if we use $\eps=m/m_1$ as mass parameter, all the three equal masses are zero in the limit $\eps=0$ and their mutual influences
vanish. Hence, the problem reduces to three copies of the 2-body problem,
formed by the mass at the north pole and a body of zero mass. The changes in this highly degenerate situation for small $\eps>0$ will be
studied in Section \ref{sect:pertur}.

\subsection{\bf The solutions with $z=0$} \label{subs:z=0}

As we are also interested in the behaviour of orbits for small $z>0$, it is also necessary to consider the solutions with $z=0$. Skipping the trivial factors, we obtain again the limit equation $T(M)=0$ for all $\Gamma$. Again, this fact is not  surprising because, for any positive $\Gamma$, we have that $\Omega\to\infty$ when $z\to 0$ and, therefore, the relative equilibrium requires larger and larger angular velocity. This means that the centrifugal force and the reaction of the constrains that keep the bodies on $\mathbb S^2$ are so large that the attraction of the mass lying at the north pole can be neglected.

\vspace*{2mm}
Regarding the cases in Subsections \ref{subs:1+3} and \ref{subs:z=0}, we will further consider the behaviour of the branches emerging from the solutions of $T(M)=0$ when the parameters $\eps$ and $z$ move away from zero. This analysis is cumbersome
due to the presence of two parameters and of some long expressions. Furthermore, when $\eps$ tends to zero, we want to study arbitrary values of $z \in (0,1)$
and, when $z\to 0$, to consider arbitrary  values of $\Gamma$ in $(0,\infty)$. The bifurcations that occur in these cases will be studied in Section \ref{sect:pertur}.

\section{Numerical experiments} \label{sect:numeric}

The results of this section have been obtained using the polynomial $p(\mu)$ computed symbolically with PARI. According to the notation and reductions introduced above, we will also refer to this polynomial as $\hat{p}(M,\Gamma, z,D(z))$.

For given values of $\Gamma$ and $z$, we first computed the zeroes $M_1,\ldots,M_6$ of the polynomial $\hat{p}(M,\Gamma,z,D(z))$. We used for the results plotted here a variable number of decimal digits, going up to 100 or more, and  performed many  additional checks.

Recall that the complex zeroes, $M$, correspond to values of $\mu$ of the form $\pm\alpha \pm\ii\beta$, called complex saddles (CS); the real positive zeroes, giving values  $\pm\alpha$ for $\mu$, are called real hyperbolic (H); and the negative zeroes, yielding $\pm \ii\beta$ for $\mu$,  are called elliptic (E).
Changes in the stability properties occur when the zeroes pass from one type to another. The exceptional cases in which some zeroes of $\hat{p}(M,\Gamma,z,D(z))$ are equal to zero or negative and coincident deserve attention to decide about the spectral stability of the solution, but they generically occur only in a zero-measure subset of $(\Gamma,z)$.

We will further use the coding ${\mbox E}^i{\mbox H}^j{\mbox {CS}}^k$, where the
exponents show the number of zeroes,  $M$, 
of each type. Of course, the exponents satisfy the identity $i+j+2k=6$. 
In Figure \ref{fig:stabreleq} we display some numerical results. In the electronic version of this paper, the colour coding is
\[ {\mbox E}^6 \!\rightarrow\! {\mbox {red}},\;{\mbox E}^4{\mbox {CS}}^1
\!\rightarrow\! {\mbox {green}},\; {\mbox E}^2{\mbox {CS}}^2 \!\rightarrow\!
{\mbox {blue}},\; {\mbox E}^5{\mbox H}^1\!\rightarrow\! {\mbox {magenta}},\;
{\mbox E}^3{\mbox H}^1{\mbox {CS}}^1 \!\rightarrow\! {\mbox {pale blue}}. \]
Hence, the observed transitions correspond to two types of bifurcations:
\begin{itemize}
\item
Hamil\-tonian-Hopf (for red $\rightarrow$ green, green $\rightarrow$ blue, and
magenta $\rightarrow$ pale blue) and
\item 
elliptic-hyperbolic (for red $\rightarrow$
magenta and green $\rightarrow$ pale blue). 
\end{itemize}
The white zones are related to
forbidden $(\Gamma,z)$ domains, which correspond to $\Omega^2<0$. In the printed version of this paper, the colours translate into grey shades as follows: red = black; blue = dark grey; pale blue = grey; magenta = light grey; green = very light grey.
\begin{figure}[ht]
\begin{center}
\begin{tabular}{rr}
\epsfig{file=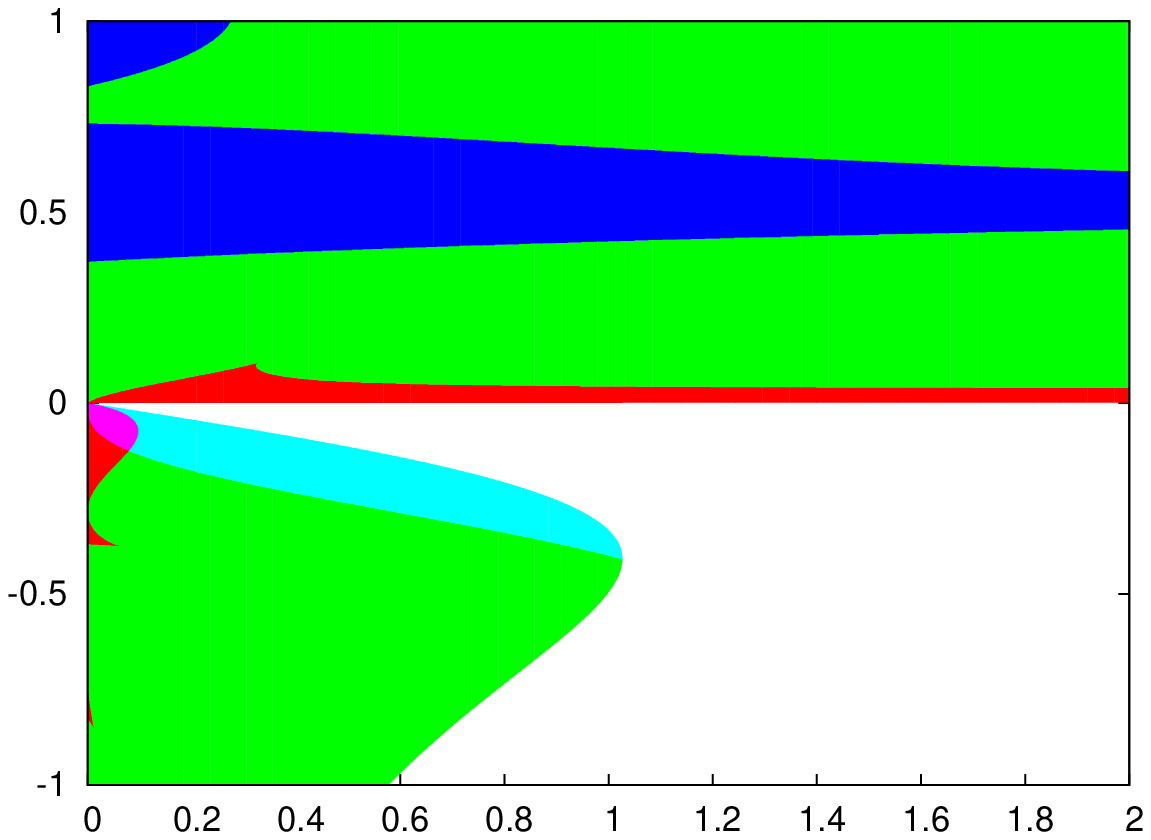,width=7.7cm} &
\hspace{-14mm} \epsfig{file=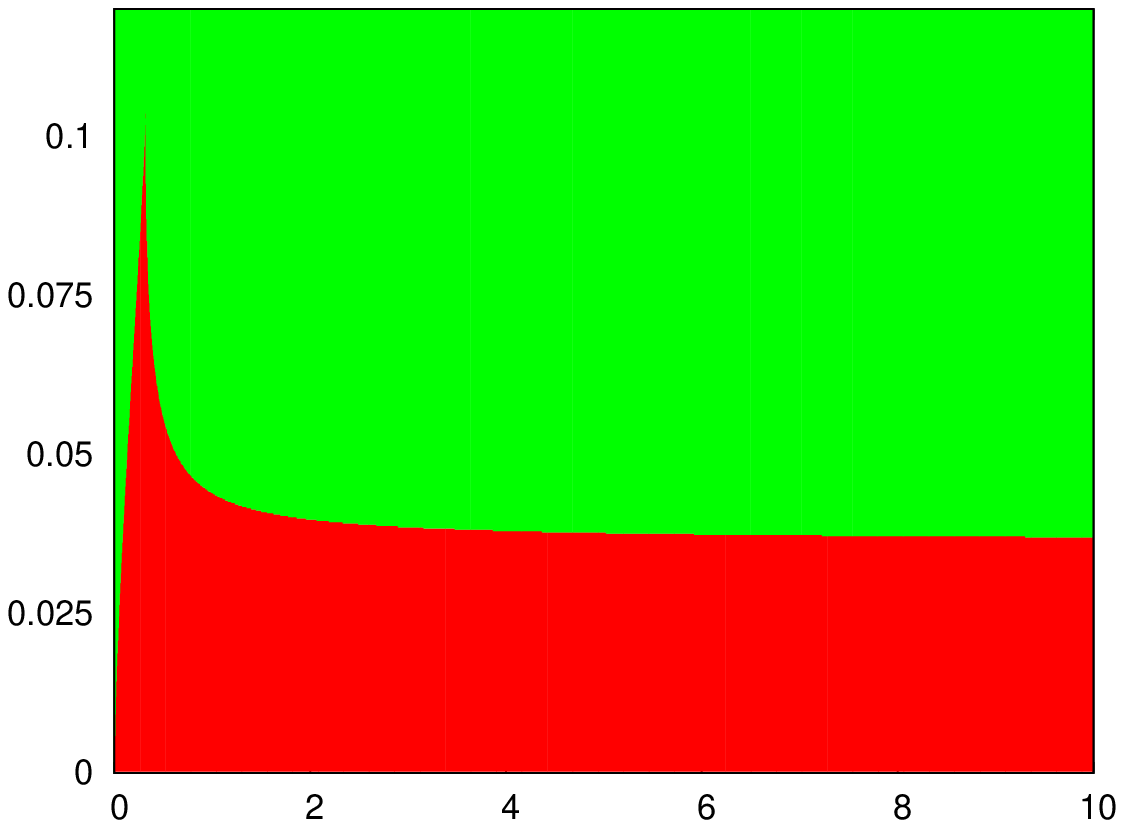,width=7.9cm} \\
\epsfig{file=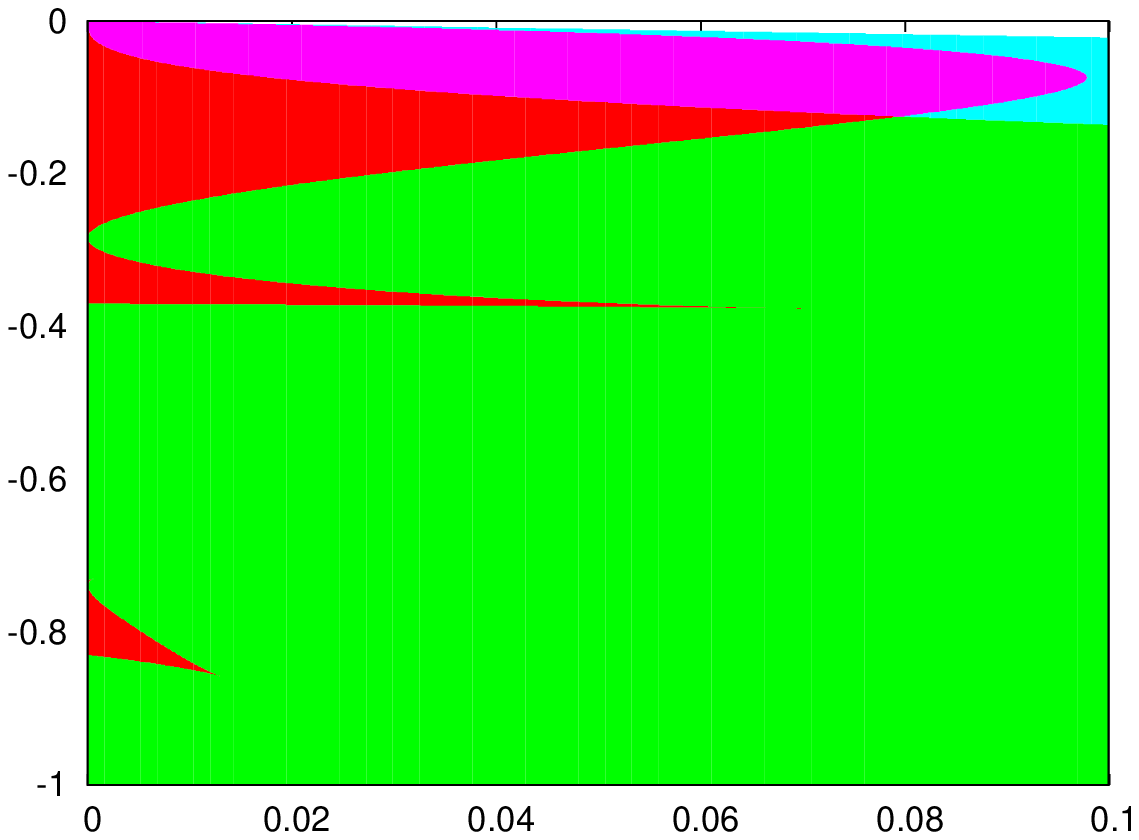,width=7.8cm} &
\hspace{-14mm} \epsfig{file=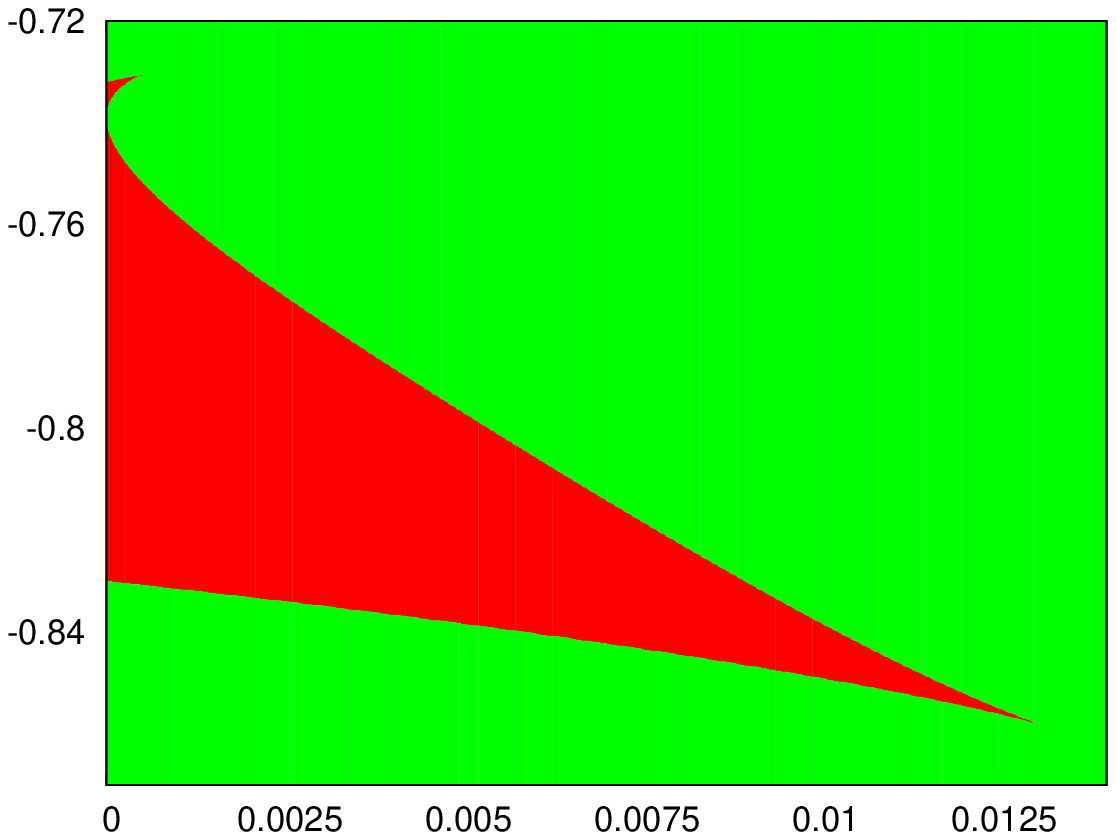,width=7.7cm}
\end{tabular}
\end{center}
\vspace{-2mm}
\caption{Regions of stability for the relative equilibrium orbits given as a
function of the mass ratio $\Gamma$ for the horizontal variable and $z$ for the vertical
variable (see the color code in the text). Top left: a general plot for $\Gamma\in
(0,2]$. Top right: magnification for $\Gamma\in(0,10],z\in(0,0.12]$. At the
bottom we show a magnification of a narrow range $\Gamma\in(0,0.1]$ for $z<0$
(left part) and an additional magnification around the little red
triangle-shaped domain which can be seen near $\Gamma=0,z=-0.8$.}
\label{fig:stabreleq}
\vspace{-2mm}
\end{figure}

We first describe the case $z>0$. The plot shows that for small $\Gamma$ the orbit
is unstable, except that near $z=0$ there is a line, emerging from $\Gamma=0, z=0$, where a super-critical Hamiltonian-Hopf bifurcation occurs and the system becomes totally elliptic. Again, for small $\Gamma$ we find different zones for which one or two CS show up. As $\Gamma\to 0$, the values of $z$ at which the transitions occur tend to
\[ z_1\approx 0.8299852976470169,\quad z_2\approx 0.7318602978602651,\quad
 z_3\approx 0.3702483631504248,\]
which correspond to the values $r_1,r_2$, and $r_3$ (see Section \ref{sect:limitm10})
found in \cite{masimo} for the Lagrangian orbits of the curved 3-body problem in 
$\mathbb S^2$.

When $\Gamma$ increases, as seen in the top left plot, a narrow red stable
domain seems to persist near $z=0$. The top right plot suggests that this is
true up to $\Gamma=10$. For some larger values of $\Gamma$, up to $10^3$, this
estimate seems to be still true. We can further ask about the limit behaviour when $\Gamma\to\infty$. The numerical evidence suggests, on one hand, that the limit value of $z$ up to which the solution is totally elliptic is close to $0.03642$; on the other hand, the boundary of one of the blue domains goes to $z=1$ and the domain disappears. The intermediate blue domain seems to shrink. Figure \ref{fig:zposglob} provides more information: the blue domain shrinks to a point and increases again to merge with another blue domain born near $\Gamma=2.91, z=0.822$. It is remarkable that to the left of that point a tiny totally elliptic zone appears (one has to magnify the plot to see it). The blue domain for large $\Gamma$ seems to tend to a limit width confined by values approaching $0.5$ and $\approx 0.94215$.

\begin{figure}[ht]
\vspace{-2mm}
\begin{center}
\begin{tabular}{cc}
\hspace{-5mm}\epsfig{file=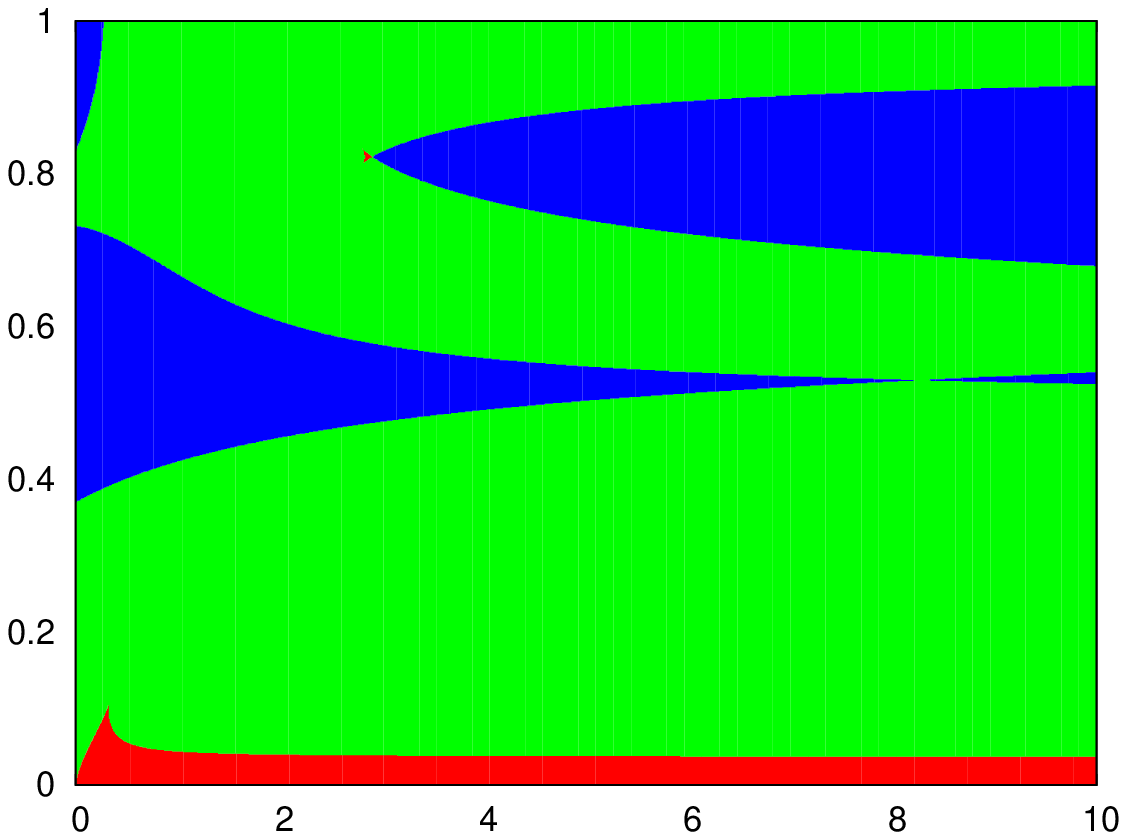,width=7.7cm} &
\hspace{-5mm}\epsfig{file=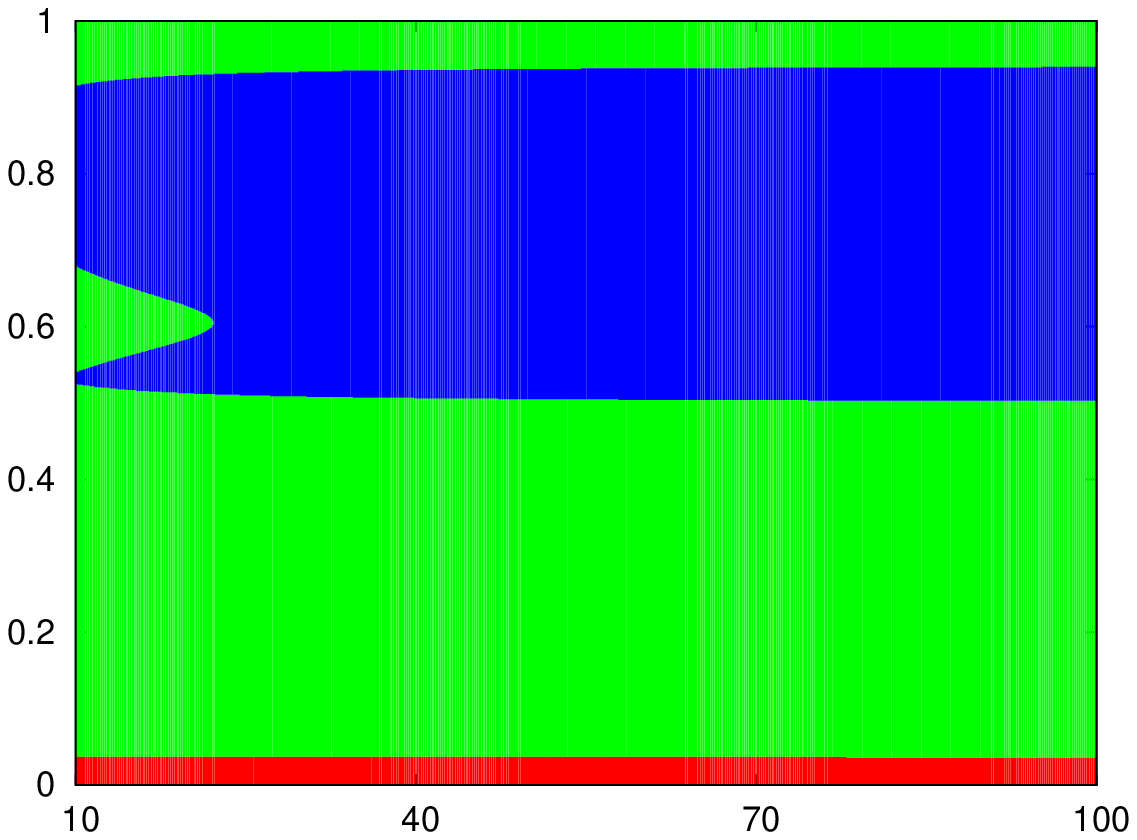,width=7.7cm}
\end{tabular}
\end{center}
\vspace{-2mm}
\caption{Part of the domain $z>0$ for large values of $\Gamma$. The colour codes are
described in the text.}
\label{fig:zposglob}
\end{figure}

These numerical experiments raise the following theoretical questions for $z>0$:
\begin{itemize}
\item [(a)] What happens when $\Gamma\to\infty$ (i.e.\ when $ \varepsilon \to 0$)? 
Do the red, green, and blue
zones in Figure \ref{fig:zposglob} on the right tend to a limit?
\item [(b)] What happens for $z$ very close to zero? In that case the value of
$\Omega$ tends to $\infty$ when $z\to 0$ and the limit is singular.
A priori, some changes cannot be excluded in a tiny strip.
\item [(c)] Which is the local behavior for $\Gamma, z$ when both are positive and close
to 0?
\end{itemize}

We will return to these questions in Section \ref{sect:pertur}.
\bigskip

We further consider the case $z<0$. As we already explained, the value of $\Gamma$ is 
bounded by the condition $\Omega^2=0$. In other terms, the boundary is parametrized by 
$z\in(-1,0)$ as
\[ \Gamma(z)= \Gamma^*(z)= -8z/[\sqrt{3}(1\!+\!3z^2)^{3/2}],\]
which we can also write as $\Gamma^*(z)=-9z(1+3z^2)D/4$. The most
interesting domains appear for small values of $\Gamma$. There are two ranges of $z$, namely $(-z_1,-z_2)$ and $(-z_3,0)$, in which the necessary conditions for linear stability are satisfied, in agreement with Proposition \ref{prop:rest} and the corresponding results obtained in \cite{masimo} for Lagrangian solutions. These ranges extend to small values of $\Gamma$; but there are two exceptions (both shown at the bottom of Figure \ref{fig:stabreleq}), namely when the axis $\Gamma=0$ is tangent to the red domains. These tangent points are located near $z=-0.73892,\,z=-0.28396$, and correspond to Hamiltonian-Hopf bifurcations. A red $\to$ magenta transition is seen ending on a tangency to the vertical axis at $(0,0)$. The transition from the magenta to the pale blue domain seems also to be very close to the boundary $\Gamma=\Gamma^*(z)$ of existence of admissible values of $z$.

These numerical results lead to the following problems in the case $z<0$:
\begin{itemize}
\item [(d)] Prove that the transitions from stability to instability that occur in the restricted
problem persist for $\Gamma>0$.
\item [(e)] Prove that there are exactly two additional values of $z$ for which a
curve of  Hamiltonian-Hopf bifurcations is tangent to $\Gamma=0$.
\item [(f)] Analyze the vicinity of $(\Gamma,z)=(0,0)$ for $z<0$.
\end{itemize}
Like the questions (a), (b), and (c), we will address these problems in the next section.

\section{The perturbation of the limit cases and the main result}\label{sect:pertur}

In this section we prove several results concerning perturbations of limit
cases. The conclusions are summarized in Subsection 6.5. All proofs are
analytical. The only use of some numerical information appears in the
computation of the zeroes of a few polynomials of the form $H(z,D(z))$, a procedure that can be reduced to computing the zeroes of irreducible polynomials in $z$ or by checking that some polynomials have a given sign at a given value of the variable. When we check that some polynomial is zero at a zero of some function, we either use the resultant or compute the zero with increasing number of digits. If $d$ decimal digits are used and the zero is simple (respectively double), we check that the obtained value is zero up to approximately $d$ (respectively $d/2$) digits. We increase $d$ up to a value that exceeds 1000.

We begin with a lemma about the double zeroes of a function $f(x,a,b)$, which
depends nontrivially on two parameters $a$ and $b$, i.e.\ neither
$f_x$ nor $f_a$ nor $f_b$ are identically zero. In the applications to the
present problem, $x$ corresponds to the variable $M$, whereas $a$ and $b$ to $z$
and $\Gamma$, respectively. We would like to see, for instance, if, for fixed $\Gamma$, two real negative zeroes of $\hat{p}$ that collide at a given value of $z$ move away from the real axis, as well as what happens when $\Gamma$ changes. The information we obtain is only based on the properties of $f$. We could exploit the fact that we are dealing with eigenvalues of an infinitesimal symplectic matrix (or a matrix conjugated to it), but some singular limit behaviour, such as when $(\Gamma,z)\to(0,0)$, makes difficult to analyze perturbations of the limit case. Since we are interested in the vicinity of a point $(x^*,a^*,b^*)$, we shift the origin of the coordinate system to that point. We can now prove the following result.

\begin{lemma} Let $f(x,a,b)$ be a real analytic function depending on the
parameters $a,b$. Assume that for $a=b=0$ the function has a zero of exact
multiplicity $2$, located~at $x=0$, i.e.\ $f(0,0,0)=f_x(0,0,0)=0$ and, for
concreteness, $f_{xx}(0,0,0)>0$. We~want to study the behaviour of $f$ in
a neighbourhood of $(0,0,0)$. For fixed $b=0$, we have:
\begin{itemize}
\item [(i)] If $f_a(0,0,0)>0$ when $a$ increases, crossing the value $a=0$, the
roots move away from the real axis. The case $f_a(0,0,0)<0$ is similar when $a$
decreases.
\item [(ii)] If $f_a(0,0,0)=0$, consider $f_{aa}(0,0,0)$ and $f_{xa}(0,0,0)$. If
the discriminant $f_{xa}^2-f_{xx}f_{aa}$ at $(0,0,0)$ is positive, the roots
remain real.
\end{itemize}
Let now $b$ vary. Then:
\begin{itemize}
\item [(a)] Under the assumptions of $({\rm i)}$, there exists a line
$a=h(b)$ along which $f$ has double zeroes in the $x$ variable, and when $a$
increases, crossing the value $a=h(b)$, the roots move outside the real axis.
\item [(b)] Under the assumptions of $({\rm ii)}$, if $f_b(0,0,0)>0$, there exists a curve $b=k(a)$, with positive quadratic tangency to $b=0$ at $a=0$, such that the zeroes of $f$ pass from real to complex when crossing the line
$b=k(a)$.
\item [(c)] Under the assumptions of $({\rm ii)}$, and if $f_b(0,0,0)=0$,
there are two curves, say $h_1(b),h_2(b)$ (eventually complex or coincident),
tending to $(a,b)=(0,0)$ when $b\to 0$. If they are real and distinct, say
$h_1(b)<h_2(b)$, then the zeroes of $f$ are real if $a<h_1(b)$ or $a>h_2(b)$ and
complex if $a\in(h_1(b),h_2(b))$. 
\end{itemize}
\label{thelemma}
\end{lemma}

\begin{proof}
The cases (i) and (ii) are elementary, since the Newton polygon
in $x,a$ involves the vertices $(2,0)-(0,1)$ and $(2,0)-(0,2)$, respectively.

To prove (a) we can assume that $f_{xx}(0,0,0)=1,f_a(0,0,0)=1$, scale the variables, and write $f(x,a,b)=x^2+a+\cO(x^3,ax,a^2)+bg(x,a,b)$. To find a double zero,
we can use the Implicit Function Theorem to express $x$ as a function $x=\hat{x}(a,b)$ from the equation $f_x(x,a,b)=0$. Inserting this $x$ in the equation $f(x,a,b)=0$, we obtain a relation between $a$ and $b$. The Implicit Function Theorem allows us then to express $a$ as a function of $b$.

To prove (b), we scale the variables and apply a linear change in the $(x,a)$
variables, after which we can write that $f(x,a,b)=x^2-a^2+\cO(|(x,a)|^3)+b(1+\hat{g}(x,a,b))$. From the equation $f_x=0$ we obtain $x=\hat{x}(a,b)$ as in (a), and if we insert $x$ in $f$, we can write $b$ as a function of $a$ that starts with a positive quadratic term in $a$. Undoing the linear change and scalings simply deforms the picture linearly.

To prove (c), we proceed as before, obtain $x=\hat{x}(a,b)$
and insert it in $f$. But now the linear term in $b$ is absent, while there is a
nonzero quadratic term in $a$. The existence of the two branches follows by
using a Newton polygon in $a,b$.  
\end{proof}

\begin{remark} In the exceptional case of item (c) in which $h_1(b)=h_2(b)$,
the function $f$ can be written, after an eventual shift of $x$, as
$f(x,a,b)=(x^2-\hat{h}(a,b)^2)\hat{f}(x,a,b)$ with $\hat{f}(0,0,0)\neq 0$, and
$\hat{h}(a,b)=0$ if the parameters $a$ and $b$ satisfy $a=h_1(b)$. Then, for 
values $a,b$ with $\hat{h}(a,b)=0$, $x$ has a double zero and the relations
$f(0,a,b)=0, f_x(0,a,b)=0, f_a(0,a,b)=0,$ and $f_b(0,a,b)=0$ hold.

In particular, the resultant of $f=0, f_x=0$ with respect to $x$, gives
$\hat{h}=0$ with multiplicity 2. However, the resultant of $f_x\!=\!0,f_b\!=\!0$
gives $\hat{h}=0$ with multiplicity 1. Note that in the case (a) the resultant
of $f_x=0,f_b=0$ is far from zero along $\hat{h}=0$. In the case (c) with 
$h_1(b),h_2(b)$ real and distinct, the resultant of $f_x=0,f_b=0$ gives a single
line with $b$ as a function of $a$ with multiplicity 1.
\label{remafterlem}
\end{remark}

\subsection{\bf Analysis of the case $\Gamma\to \infty$}

In this case only $z >0$ has sense. 
Let us consider $\hat{p}(M)$ in terms of $ \varepsilon = 1/\Gamma$.
The term in $\eps^0$ is $M^2(M+1)^4$, whereas the
terms in $\eps^j, j>0$, are polynomials of degree 6 in $M$ that, in turn, have  polynomials in $z,D(z)$ as coefficients.

To discuss how the double zero $M=0$ bifurcates as a function of $\eps$, we
compute the Newton polygon in the $\eps,M$ variables. After simplifying by a
numerical factor, we obtain
\[ 16M^2+\eps M((1728z^5-720z^3+72z)D-192z^4+48z^2)+
\eps^2[(46656z^{10}-38880z^8+11988z^6\]
\[-1620z^4+81z^2)D^2+(-10368z^9+6912z^7-1512z^5+108z^3)D+576 z^8-288z^6+36z^4]
\!=\!0,\]
a quadratic equation in $M$ with discriminant zero. Hence the dominant term of $M$ is
of the form
\[ M_0(\eps)=\frac{3\eps}{4}z(4z^2-1)h(z,D(z)),\quad h(z,D)=(3-18z^2)D+2z.\]

It is easy to check that the factor
$h$ in $M_0(\eps)$ is positive in the interval $z\in(0,1)$. Indeed, $h$ is
positive at $z=0$ and at $z=1$ and $dh/dz$ has only two zeroes in the interval $(0,1)$ at which $h$ is positive. Therefore the only value of $z$ at which $M_0(\eps)$
changes sign in the interval $(0,1)$ is at $z=z_{5,0}=1/2$. Hence, $M_0(\eps)$ becomes positive for $z>1/2$, giving rise to instability, in agreement with the lower bound of the blue domain for large $\Gamma$ obtained in Section \ref{sect:numeric}.
However, the analysis up to now shows that the two branches emerging from $M=0$
are real and coincide. It could happen that higher order terms take them away
from the real axis even for values of $z\in (0,1/2)$.

As usual, we introduce another variable $N$ by the transformation $M=M_0(\eps)+\eps N$. 
Substitution into the characteristic equation and division by $\eps^2$ gives the
dominant terms in the new Newton polygon. They turn out to be, 
up to a numerical factor, of the form
$$ N^2+\eps 3^3 z^5(z^2-1/4) g(z,D), \ g(z,D)=h(z,D)k(z,D)^2,\ k(z,D)=D
(9z^3-6z)-1. $$ 
We note that  $g(z,D)$ is positive. Therefore
\[ M(\eps)=M_0(\eps)\pm\eps^{3/2} \sqrt{3^3 z^5 (1/4-z^2) g(z,D)}. \]
This shows that the roots $M$ evolving from zero are real, negative, and distinct
for $z<1/2$, and are complex with positive real part for $z>1/2$. When the variable $z$ crosses the value $z=1/2$, a Hamiltonian-Hopf bifurcation occurs.

Let us analyze the solutions evolving from the quadruple solution $M=-1$. 
We introduce a new variable,
which we denote again by $N$, such that now $M=-1+\eps N$. The lower order term
in $\eps$ is a term in $\eps^4$ whose coefficient $Q(N,z)$ is a polynomial of
degree 4 in $N$ with polynomials in $z,D(z)$ as coefficients.

First we would like to determine the behaviour of the function for $z$ small. 
The dominant terms are
\[8N^4-18z\alpha N^3+108z^3\alpha N^2+1701z^6\alpha^2 N-78732z^9\alpha^3, \]
where, we recall, $\alpha=D(0)$. A Newton polygon method tells us that the
dominant terms in the solutions are
\[N_1=9z\alpha/4,\quad N_2=6z^2,\quad N_3=-36z^3\alpha,\quad N_4=81z^3\alpha/4.
\]
In particular all the solutions are simple and  negative, ensuring local spectral
stability near $z=0$.

To study the behaviour of the function for larger values of $z$, we compute the resultant
of $Q(N,z)$ and $\frac{d}{dN}Q(N,z)$. After skipping some powers of $z$ and $D$,
we have a polynomial of degree 40 in $z$ and degree 10 in $D$. It is easy to
check that this polynomial has only two simple zeroes for $z\in(0,1)$, located at
\[z_{4,0}\approx0.036420258329089021,\quad z_{6,0}\approx0.942152758989663983.\]
At $z_{4,0}$ a couple of roots meet and become complex, whereas at $z_{6,0}$
these roots return to the real domain.

Hence, we can now summarize the bifurcations for $\eps$ small as follows.

\begin{proposition} \label{prop:epsmall}
For $\eps=1/\Gamma$ tending to zero, there exist three functions, $z_4(\eps),
z_5(\eps)$, and $z_6(\eps)$, tending to $z_{4,0}, z_{5,0}$, and $z_{6,0}$,
respectively, at which Hamiltonian-Hopf bifurcations occur. They are sub-critical
at $z_5(\eps)$ and super-critical at $z_4(\eps)$ and $z_6(\eps)$.
Therefore the character of the fixed points is of type ${\rm E}^6$ for $z\in(0,z_4(\eps))$,
of type ${\rm E}^4{\rm CS}^1$ for $z\in(z_4(\eps),z_5(\eps))\cup(z_6(\eps),1)$, and
of type ${\rm }E^2{\rm CS}^2$ for $z\in(z_5(\eps),z_6(\eps))$.
\end{proposition}
\begin{proof}
The above analysis and items (i) and (a) of Lemma \ref{thelemma} complete the proof.
\end{proof}

\subsection{\bf The case of small positive $z$}

To study this case it is convenient to replace $D(z)$ by $\alpha
(1+3z^2)^{-5/2},$ as done before, and to expand the binomial up to the required
order. First we study the solutions emerging from $M=0$.

We proceed as in the previous subsection, by regarding $M$ as a function of $z$
for $z$ around 0. First we find two branches that coincide at order 1 in $z$: $M=-9\alpha\eps z/4$, where we use again $\eps$ to
denote $\Gamma^{-1}$. Then we seek the terms in $z^2$ that are also
coincident. At the third step there appear two branches in $z^{5/2}$ with
opposite signs. Summarizing, the solutions evolving from $M=0$ are
\[M(z)=-9\alpha\eps z/4+(81\alpha^2\eps^2-24\eps)z^2/16\pm 9(\alpha\eps^3)^{1/2}
z^{5/2}/2.\]
That is, the two values $M_1,M_2$ emerging from $M=0$ are real negative and
they only differ in the $\cO(z^{5/2})$ terms. For further reference we denote
them by $M_1$ (with $+$) and $M_2$ (with $-$).

For the solutions that evolve from $M=-1$, we write $M(z)=-1+N$ and
compute the Newton polygon in the $z,N$ variables. 
After simplifying constants, the dominant terms in the polygon are
\[4N^4-9\alpha zN^3\varepsilon+54\alpha z^3N^2\varepsilon^2+3584z^6N\varepsilon^3-
165888\alpha z^9\varepsilon^4.\]
From this expression we obtain the dominant terms of the four branches, already
separated at this first step,
\[ N_1= 9\alpha\eps z/4,\quad N_2=6\eps z^2,\quad N_3=256\eps z^3/(3\alpha),
\quad N_4=-4096\eps z^3/(27\alpha).\]

We can now summarize the above results as follows.

\begin{proposition} \label{prop:zsmall}
For any fixed value of the mass ratio $\Gamma$ in $(0,\infty)$, there is a range
of values of $z$ close to zero, the upper limit of the range depending on
$\Gamma$, such that the spectral stability of the orbit is preserved.
\end{proposition}

This result explains the red domain displayed in the previous figures
for $z >0$ small. Note, however, that, as soon as $\Gamma\to 0$, we have $\eps\to\infty$, and then the range of validity of Proposition \ref{prop:zsmall} is not uniform, 
since it can go to zero when $\Gamma\to 0$. This fact has been already put into the evidence in the Figures at the corner near $(\Gamma,z)=(0,0)$ in the first quadrant, where a bifurcation line is seen to emerge from $(0,0)$. The required analysis follows in next subsection.

\subsection{\bf Study of the vicinity of $(\Gamma,z)=(0,0)$ for both $z>0$ and
$z<0$} \label{sect:near00}

When approaching $(0,0)$ in the $(\Gamma,z)$-plane, we have a
singular problem. Depending on the direction, the value of $\Omega^2$ can tend
to any real non-negative value. Therefore, before proceeding with the analysis, we
must add a short description of the difficulties we face, based on the following numerical experiment.

For this purpose, we wanted $\Gamma$ to be neither too small (to exclude a configuration with too close lines), nor too large (to be inside the domain of interest), and thus chose $\Gamma=0.03$. Then the values of the solutions $M$ were
computed as a function of $z$. The value of $z$ corresponding to $\Omega^2=0$ is
$z^*\approx -0.00649642410717306$. Figure \ref{fig:near00} on the left plots three of
the $M$ values, after multiplying them by 
$\Omega^2$, which is a measure of the distance from $z$ to $z^*$.
For the values which appear
to be almost constant and tend to coincide when $z\to (z^*)^-$, we also
changed the sign. This means that when approaching $\Omega^2=0$, one value of $M$ seems to tend to $+\infty$, whereas two values seem to tend to $-\infty$.

\begin{figure}[ht]
\begin{center}
\begin{tabular}{rrr}
\hspace{-4mm}\epsfig{file=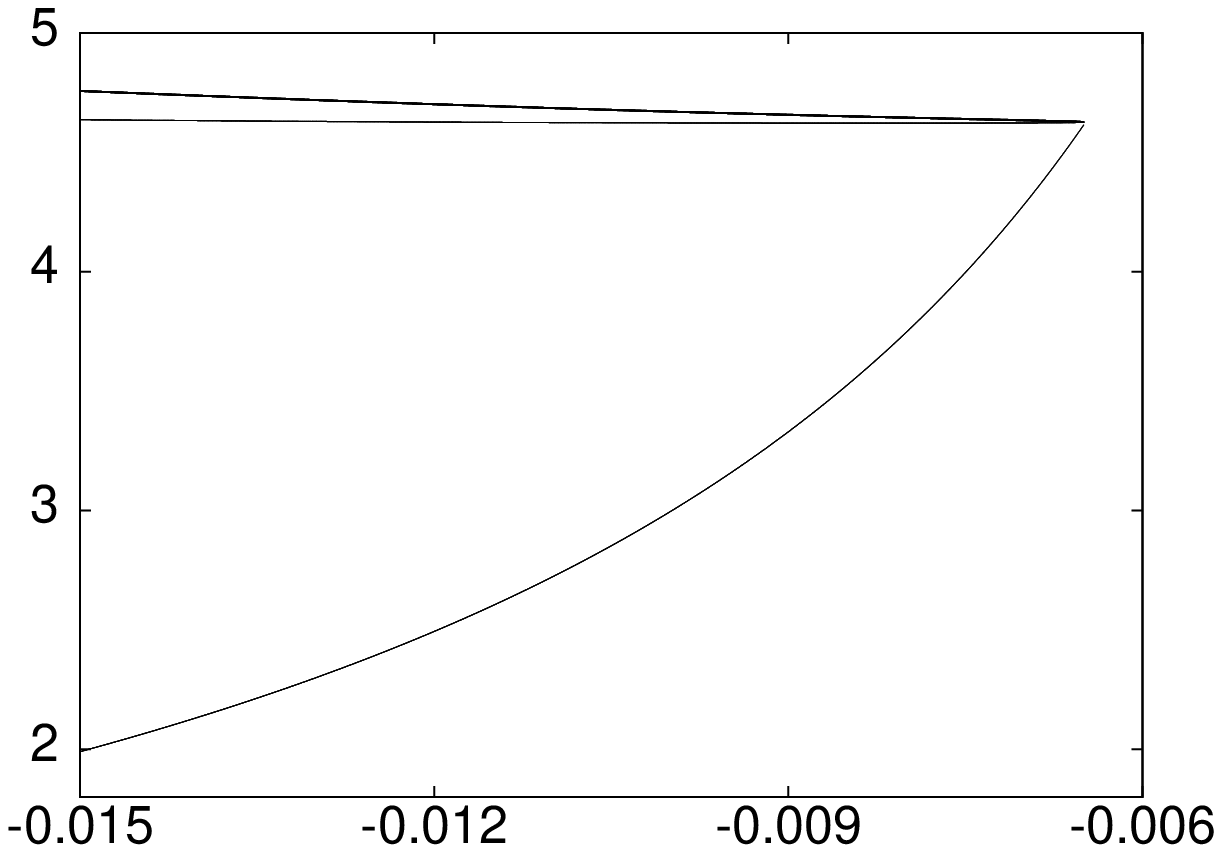,width=5.2cm} &
\hspace{-6mm}\epsfig{file=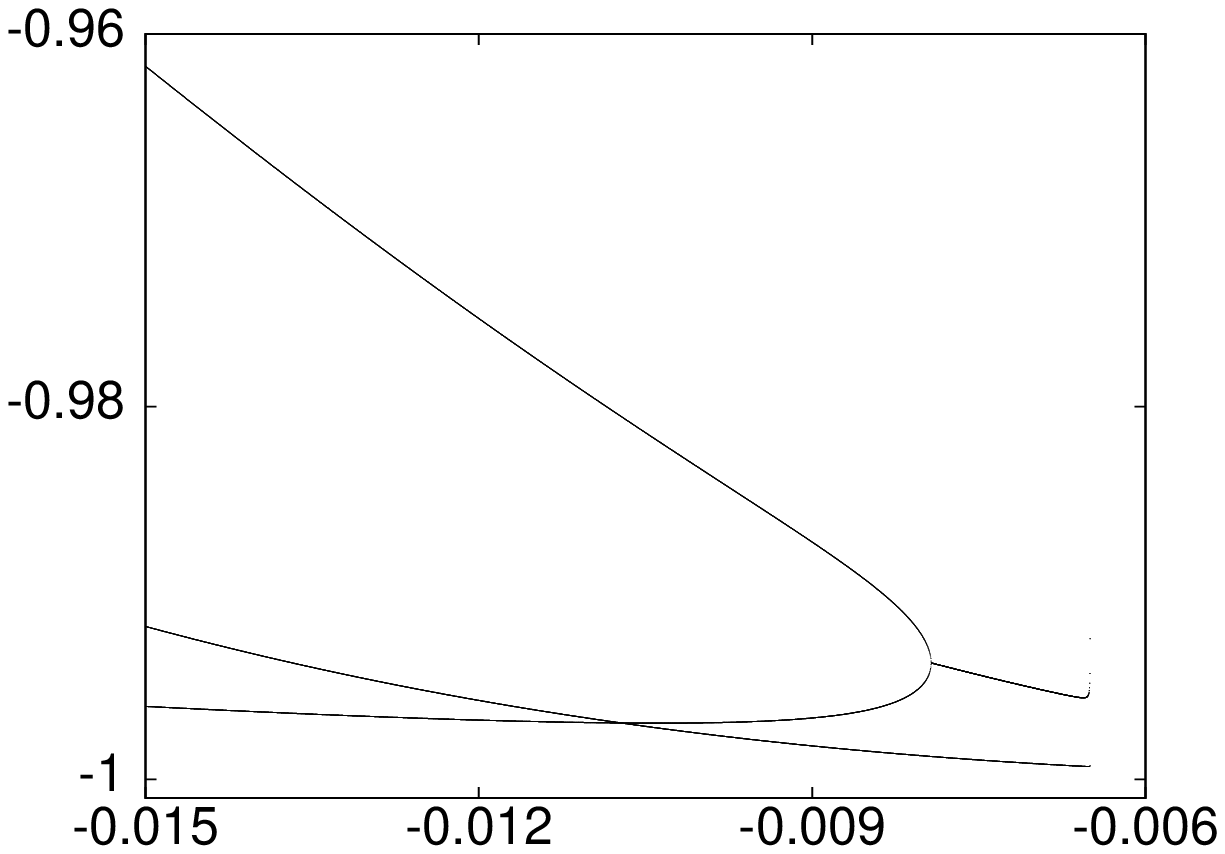,width=5.2cm} &
\hspace{-6mm}\epsfig{file=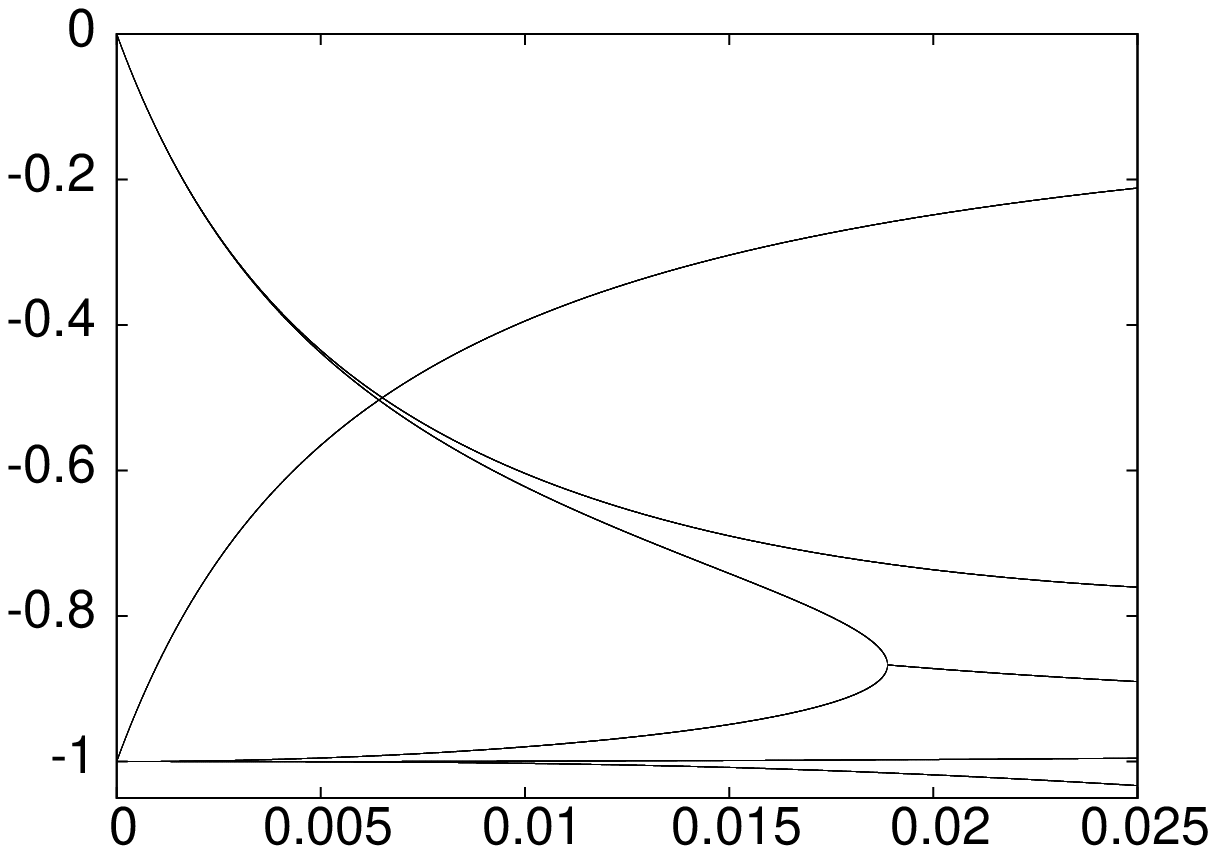,width=5.2cm}
\end{tabular}
\end{center}
\caption{Evolution of the solutions $M_i$ as functions of $z$ for $\Gamma=0.03$
as an illustration of the behaviour of the zeroes of $\hat{p}(M,\Gamma,z,D(z))$
for small $z$ and $\Gamma$. See the text for a detailed explanation.}
\label{fig:near00}
\end{figure}

The other three values are plotted in real scale on the middle
plot. We can clearly see the Hamiltonian-Hopf bifurcation near $z=-0.00793$.
After that value of $z$, we only plot the real part of these solutions.  For 
a previous value near $z=-0.01072$ there is also a double root, which seems
to avoid a bifurcation. Furthermore, a multi-precision study in narrow
intervals provides evidence that, when $z$ comes very close to $z^*$, the
solutions that became complex have a real part that turns to positive (without
giving rise to a bifurcation), whereas the negative one remains negative,
tending to a value, when $z\to z^*$, that tends to zero when $\Gamma\to 0$. In conclusion, we can expect two double zeroes in that domain, only one giving rise
to a bifurcation.

When we pass to values of $z>0$, the behaviour of the zeroes is shown in Figure
\ref{fig:near00} on the right. The solution $M_1$ evolving from 0 remains real while
$M_2$ provides a Hopf bifurcation when it collides with the solution that starts
at $-1+N_2$. The branches starting as $-1\!+\!N_3$ and $-1\!+\!N_4$ remain real
and stay close to $-1$. Similarly, the~solution starting as $-1+N_1$ remains
real, although it coincides with the solutions $M_1$ and $M_2$ at two values of
$z$, which are very close. Hence, we should expect three double roots for that
value of $\Gamma$ at some small $z$ values, only one giving rise to a
bifurcation.

The above numerical evidence will help us obtain some theoretical results.
But before proceeding with the analysis of the bifurcations, it is useful to study the behaviour of the function in the vicinity of the curve $\Omega^2=0$ near $(\Gamma,z)=(0,0)$. This approach presents interest in itself because it gives us the eigenvalues and allows us to interpret the results we will later obtain.

For this purpose, we follow an approach different from the one that led us to Figure
\ref{fig:near00}. Instead of fixing $\Gamma$ and allowing $z$ to vary, we fix $z<0$
near 0 and vary~$\Gamma$. Recall that $\Gamma^*(z)$ has been defined as the
limit value corresponding to $\Omega^2=0$. We further define $\gamma=\Gamma^*(z)-\Gamma$ and want to study what happens when $\gamma\!\to\!0^+$.
Our findings are summarized below.

\begin{proposition} \label{prop:nearcurve}
For $z<0$ close to zero, let $\Gamma^*(z)$ be the value of $ \Gamma$
corresponding~to $\Omega^2=0$ and $\gamma=\Gamma^*-\Gamma$. Then, when
$\gamma\to 0^+$, the roots of the characteristic polynomial, for a fixed value of $z$, and then when $z\to 0^-$, behave as follows:
\begin{itemize}
\item [(i)] Two roots are real and negative, and when multiplied by $\gamma/|z|$ they tend 
to a common value $\chi_1(z)$, behaving close to the limit, 
when $\gamma\to 0^+$,
like $\chi_1(z)\pm \cO(\sqrt{\gamma z}).$ The value of $\chi_1(z)$ tends to $-8/\sqrt{3}$ when $z\to 0^-$ with a dominant term that is linear in $z$.
\item [(ii)] A third root is real and positive, and when multiplied by $\gamma/|z|$
it tends to a value $\chi_3(z)$, from below, behaving close to the limit like
$\chi_3(z)-\cO(\gamma z).$ The value of $\chi_3(z)$ tends to
$8/\sqrt{3}$ when $z\to 0^-$ with a dominant term that is quadratic in $z$.
\item [(iii)] A fourth root is real negative and tends to a value $\chi_4(z)$
linearly in $\gamma$. The limit value $\chi_4(z)$ tends to $-1$ when
$z\!\to\! 0^-$ with a dominant term linear in $z$.
\item [(iv)] The last two roots are complex, and when multiplied by $\psi=\gamma/z^4$ they tend to a common nonzero real value. The real part multiplied by $\psi$ tends to
$\chi_5(z)$, which tends linearly in $z$ to $27/2$ when $z\to 0^-$. The
imaginary parts multiplied by $\sqrt{\psi}$ tend to values $\pm\chi_6(z)$, which
tend linearly in $z$ to $\pm \sqrt{54}$ when $z\to 0^-$.
\end{itemize}
\end{proposition}
\begin{proof} We only sketch the main steps, the full result following then easily. 
The characteristic polynomial $\hat{p}$ is written as a function of $z$ 
(a polynomial in $z$ and $D(z)$), $\gamma$, and $M$, and we look at the Newton 
polygon in the variables $z,\gamma$, both of them to be seen as small,
involving the exponents $(0,5),(1,4),(2,3),(3,2),(7,1)$ and
$(11,0)$, with coefficients that are polynomials in $M$, given by
\[ -\gamma^5 2^{10}M^2(M+1)^4+\gamma^4z\alpha 2^8 3^2 M(M+1)^3(M+2)+
\gamma^3(z\alpha)^2 2^6 3^4 (M-1)(M+1)^3- \]
\vspace{-5mm}
\[\gamma^2(z\alpha)^3 2^4 3^6 (M+1)^3+ \gamma z^7\alpha^3 2^4 3^9 (M-1)(M+1)-
z^{11}\alpha^3 2^2 3^{12} (M+1).  \]
The last equation has the obvious solution $M=-1$, which is simple. By adding the 
remaining terms in $\hat{p}$, the solution mentioned in (iii), denoted as
$M_4$, is obtained.

Using the side between $(0,5)$ and $(3,2)$, the variables $z$ and $\gamma$ should
be of the same order. This suggests to change the variable $M$ to $N$ by
$M=N z \alpha/\gamma.$ Setting $ \gamma=0$, and simplifying by some powers of
$z$ and $N$, we obtain the polynomial $64N^3-144N^2-324N+729,$ which has the roots $N=9/4$
(double) and $N=-9/4$. Hence, we obtain solutions $M_{1,2}$ whose main terms
are $9\alpha z/(4\gamma)$, as stated in (i), and $M_3$ whose main term is 
$-9\alpha z/(4\gamma)$, as stated in (ii).

Using the side between (3,2) and (11,0), we obtain that $\gamma=\cO(z^4)$, which suggests a
change of variable from $M$ to $N$ defined as $M=Nz^4/\gamma$. As before, setting
$\gamma=0$, and simplifying the powers of $z$, we obtain the polynomial $4N^2-108N+729,$ which has
the roots $27/2$ (double). Hence, we obtain solutions $M_{5,6}$ whose main
terms are $27z^4/(2\gamma)$, as stated in (iv).

Let us denote by $M^{(0)}_i, i=1,\ldots,6,$ the first approximations of the 6 solutions
found up to now. As usual, we write $M_i=M^{(0)}_i+\Delta M^{(0)}_i$ and
substitute them in the initial equation. We then obtain the new Newton polygons and find the corrections, as described in the statement.
\end{proof}

\begin{remark}
The properties described in Proposition \ref{prop:nearcurve} agree with the observed fact that the points close, but below, the line $\Omega^2=0$ belong to the pale blue domain.
\end{remark}

We return now to study a vicinity of $(\Gamma,z)=(0,0)$. We know that the bifurcations we are looking for are associated to $M=0$ or to double roots. We begin with the case $M=0$.

By skipping a suitable factor, the polynomial $\hat{p}(0,\Gamma,z,D(z))$ has
\[ 32 \Gamma^3+144Dz\Gamma^2+162D^2z^2\Gamma+3645 D^3z^5\]
as Newton polygon in the $z,\Gamma$ variables, where its coefficients still depend on $D(z)$.

The last two terms give the branch whose dominant term is $\Gamma=-45Dz^3/2$, which
can be written around $z\!=\!0$ as $\Gamma\!=\!-45\alpha z^3/2\!=\!-80z^3/\sqrt{3}$, in
perfect agreement with the numerical results. This phenomenon is easily identified with the red to magenta transition in Figure \ref{fig:stabreleq} on the left, both top and bottom. The other vertices give rise to a factor $\Gamma=-9zD(1+3z^2)/4+27z^4/
2$, double up to order 4 in $z$, but which is located between $z=0$ and the
curve $\Omega^2=0$ and, hence, outside the admissible domain.

We will further study the double roots. As mentioned at the end of Section
\ref{sect:equ}, the resultant polynomial Res$(\Gamma,z,D(z))$ is huge, but for
$(z,\Gamma)$ near $(0,0)$ it is still feasible to compute the Newton polygon.
The relevant vertices bounding the three sides of the polygon have exponents
$(16,0),(12,2),(9,4)$, and $(0,13)$. After simplifications, the first side from $(16,0)$ to $(12,2)$ leads to branches with dominant terms given by the solutions of the equation $324z^4\alpha^2+\Gamma^2=0$. They are complex and can be discarded.

The second side, with endpoints $(12,2),(9,4)$, gives the condition for the
dominant terms $54z^3\alpha-\Gamma^2=0$, with the real solution
$z=(4\times3^{1/6})^{-1}\Gamma^{2/3}$. It is easily identified as the curve
that separates the green and red domains from each other in Figure \ref{fig:stabreleq}
top, both left and right, near $(0,0)$.

The third side gives branches with dominant terms of the form $\Gamma=9z\alpha/
4, \Gamma=-9z\alpha/4$, and $\Gamma=-36z\alpha/25$ of multiplicities 4, 3, and 2,
respectively. We begin with the case of multiplicity 3. Setting
$\Gamma=z(-9D(z)/4+\gamma)$ in the resultant and simplifying by constants and
powers of $z$ and $D$, we obtain the polynomial
\[ 128\gamma^3+3888 D^2 z\gamma^2+52488 D^3 z^3\gamma+177147 D^4 z^5.\]
This gives raise to one branch which, to order 2 in $z$, and using the full
$D(z)$ function, not only $D(0)=\alpha$, is of the form
$\Gamma=-z9D(z)/4-243z^2D^2/8$, easily identified as the transition from magenta to pale blue. The other root is double, $\gamma=-27 z^2 D/4$. At the next step, $\Gamma=z(-9D(z)/4)+z^3(-27D/4+\psi)$, we obtain again a double solution, $\psi=27z/8$. But the first part of $\Gamma$ given by $-9zD(z)(1+3z^2)/4$ is the
boundary $\Omega^2=0$. Hence, the obtained double branch is already outside the admissible domain at order 4 in $z$.

We now consider the branch beginning with $\Gamma=-36z\alpha/25$ of
multiplicity 2. In fact, the successive Newton polygons that we computed in the
expression of $\Gamma$ as a power series in $z$ always give multiplicity 2.
Hopefully this branch of double zeroes corresponds to the double zeroes that
appear in Figure \ref{fig:near00}, in the middle, and do not give rise to a bifurcation.
We further computed two additional resultants.
Up to now we are using Res$(\Gamma,z,D(z))$, obtained from the elimination of $M$
between $\hat{p}$ and $\partial \hat{p}/\partial M$. After simplification, the degrees in $z,D(z)$, and $\Gamma$ are  104, 25, and 13, respectively, as mentioned before, and the polynomial contains 6779 terms.

Let Res$_2(\Gamma,z,D(z))$ be the resultant from $\hat{p}$ and $\partial\hat{p}/
\partial \Gamma$ and Res$_3(\Gamma,z,D(z))$ the resultant from $\partial \hat{p}/
\partial M$ and $\partial\hat{p}/\partial\Gamma$. The corresponding degrees and
numbers of terms are similar (111, 27, 11, and 7453 for Res$_2$ and 96, 23, 11, and
5474 for Res$_3$). But the important thing is that the Newton polygons of Res and
Res$_2$ give, up to the computed order, a branch of multiplicity 2, while the
one of Res$_3$ is simple. More precisely, the Newton polygon of Res$_3$ has
degree 11 and factorizes as
\[ (4y - 9)^3 (4y + 9)^3 (8y - 9) (25y - 36) (25y + 36) (80y^2 + 720y + 1377),\]
where $y$ denotes the ratio $\Gamma/(zD)$. We are interested in the ratio
$y=-35/25$, simple as claimed. A few terms of the expansion of $\Gamma$ as a
function of $z,D(z)$ are obtained, in a recurrent way, as
\[ \Gamma\!=\!z(g_2-\frac{36D}{25}),\;\;g_2\!=\!z\frac{g_3\!-\!(1944D^2\!+\!54)}
{625}, \;\; g_3\!=\!z\frac{g_4\!+\!(186624D^4\!+\!31968D^2\!+\!144)}{25D},\]
\vspace{-2mm}
\[g_4=z\frac{g_5\!-\!(35831808D^6\!+\!59222259D^4\!-\!12906D^2\!+\!768)}{50D},\]
\vspace{-2mm}
\[g_5=z\frac{g_6\!+\!(3439853568D^8\!+\!10318220802D^6\!+\!42691698D^4\!+
\!226512D^2\!+\!2048)}{25D},\ldots \]

Hence, the branch is double and, according to Lemma \ref{thelemma} and Remark
\ref{remafterlem}, no bifurcation occurs along that line. As a side information
we note that along that double branch, for $z<0$ small, the value of $M$
is close to $-1$.

Finally we consider the branch starting with $\Gamma=9z\alpha/4$ of
multiplicity 4. Writing $\Gamma=z(9D/4+g_2)$ and substituting in Res, we
obtain the polynomial
\[ -16g_2^4+288 z g_2^3-1368 z^2 g_2^2+648 z^3g_2-81 z^4,\]
which factorizes as $(4g_2^2-36zg_2+9z^2)^2$. Hence, the terms of order 2 in $z$
give rise to two double solutions, with coefficients $9/2\pm 3\sqrt{2}$. As in
the previous case, keeping with Res, we obtain double solutions for the computed
next terms.

We will further use Res$_3$. As mentioned before, one of its factors is $4y-9$
with multiplicity 3. Setting $\Gamma=z(9D/4+g_2)$, we obtain the Newton polygon in the $(g_2,z)$ variables as
\[ -512g_2^3+(8748D^2+6656)zg_2^2+(-78732D^2-19584)z^2g_2+(19683D^2+4608)z^3,\]
which factorizes as $(4g_2^2-36zg_2+9z^2)(-128g_2+(2187D^2+512)z)$.
The last factor is irrelevant for our purposes and the quadratic factor gives
the two branches with dominant terms $g_2=(9/2\pm 3\sqrt{2})z$, which are
simple. Hence, as in the previous case, the two branches of Res are double and they give rise to no bifurcation. Additionally, we can mention that these double solutions occur for $M$ close to $-1/2$ and that from the plot in Figure \ref{fig:near00} on the right we expect them to be close.

We can now summarize the results obtained in this case as follows.

\begin{proposition} \label{prop:near00}
In a vicinity of $\Gamma=0,z=0$, for $\Gamma>0$ there are three lines giving
rise to bifurcations, all emerging from $(0,0)$:
\begin{itemize}
\item [(i)] A line of E $\to$ H transition, for $z<0$, having a cubic tangency
with the axis $\Gamma=0$.
\item [(ii)] A line of E $\!^2$ $\to$ CS transition, also for $z<0$, which has a
quadratic tangency with the line corresponding to $\Omega^2=0$.
\item [(iii)] A line of E $\!^2$ $\to$ CS transition, for $z>0$, for which
$z$ is of order $\Gamma^{2/3}$.
\end{itemize}
\end{proposition}

\subsection{\bf Analysis of $\Gamma$ near zero}

In this subsection, we need only to consider bifurcations that occur away from
a vicinity of $(\Gamma,z)=(0,0)$, since the behaviour in the neighbourhood of
this point has been already studied in the previous subsection.

For positive $z$, we must only show that the changes of stability that occur for
the 3-body problem persist when we add the small mass $m_1$. As already shown in
Proposition \ref{prop:rest}, the behaviour of the small mass gives instability for the full 4-body problem. Using the results in \cite{masimo} and of Lemma \ref{thelemma}(a), it follows that the changes
of stability of the 3-body problem persist in the case $\Gamma>0$ small.

\medskip

Next we pass to the more interesting case $z<0$. The changes of stability found
in the curved 3-body problem also persist when we pass to $\Gamma>0$, according to
Lemma \ref{thelemma}(a), and the stability of the body of mass zero does not
change the stable domains for $\Gamma=0$. 
However, new changes can occur when passing from $\Gamma=0$ to $\Gamma>0$ if 
some of the additional zeroes of the form (\ref{zeropm}) for the restricted problem coincide with one of the curved 3-body problem. 

\begin{figure}[ht]
\vspace{-3mm}
\begin{center}
\begin{tabular}{rr}
\hspace{-5mm}\epsfig{file=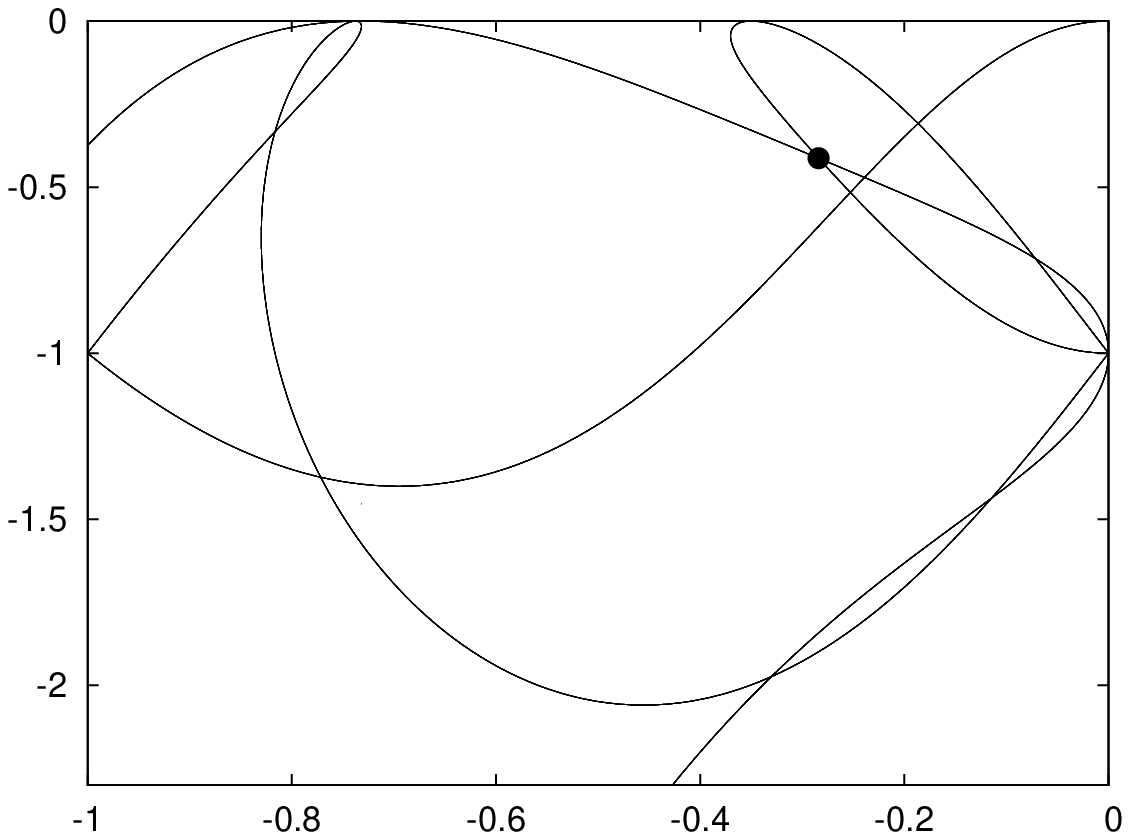,width=7.7cm} &
\hspace{-5mm}\epsfig{file=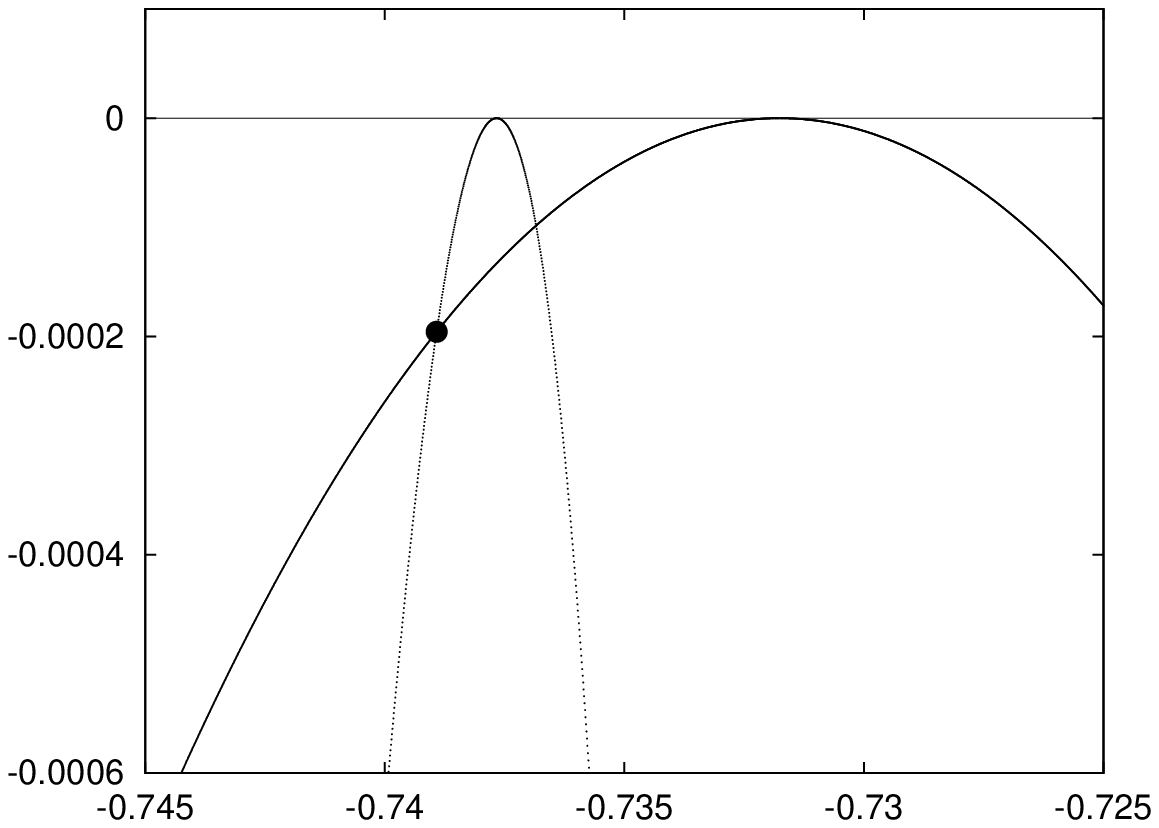,width=7.7cm}
\end{tabular}
\end{center}
\vspace{-3mm}
\caption{A plot of the negative values of $\mu^2$ as a function of $z<0$ for the
limit problem $\Gamma=0$. The large dots show the location of the double zeroes
associated with the Hamiltonian-Hopf bifurcations that can be seen as tangencies
of the red domains with $\Gamma=0$ shown in Figure \ref{fig:stabreleq} bottom,
left and right. The curves of zeroes associated with the restricted problem are
easily identified as having a tangency with $z\!=\!0$ at $\mu^2\!=\!-1$. The
branch reaching the lower boundary of the plot continues down up to $z=-1$.}
\label{fig:stabrest}
\end{figure}

Figure \ref{fig:stabrest} shows all the relevant zeroes, when real, simultaneously 
as function of $z$.
The exact solution $M_0$ (see Section \ref{sect:limit}) can be identified as the curve starting 
at $(-1,1)$ and ending at $(0,0)$ (see also formula (7) in \cite{masimo}).
Double zeroes involving $M_0$ should not be taken into consideration in the light of the
explanations given in that paper.
In Figure \ref{fig:stabrest}, we identified two
double zeroes that are the responsible for the Hamiltonian-Hopf bifurcations observed in Figure \ref{fig:stabreleq}.

To locate them, we consider the polynomial $Q(M)$ introduced in Section 
\ref{sect:limitm10}. It is convenient  to express the 
coefficients of $Q(M)$ in terms of $ Z=z^2 = 1-r^2$. We will further denote this polynomial by $Q(M,Z)$.
Also we write the equation for the additional zeroes (\ref{zeropm}) as $S(M) =0$,
where
$$ S(M) = \left( M^2 + 2 M+ \frac{z^2 G^3}{4} +1 \right)^2 - z^2 G^3 (M-1)^2. $$
Using  $ G=3 (1+3 z^2)/4 $, we can write $S(M)$ as a polynomial in $M$ with polynomial
coefficients in $Z$ to be denoted as $S(M,Z)$.
Then we compute the resultant $\cR(Z)$ of $Q(M,Z)$ and $S(M,Z)$ to eliminate $M$.  
The polynomial $\cR(Z)$, of degree 29, factorizes
as $\cR(Z)=\cR_1(Z)\cR_2(Z)$, with factors of degrees 14 and 15, respectively.
A part of the expressions is
\[ \cR_1(Z)=31381059609Z^{14}+135984591639Z^{13}+\ldots +843283683Z-2985984,\]
\vspace{-4mm}
\[ \cR_2(Z)=1162261467Z^{15} + 5036466357Z^{14} +\ldots +133996544 Z+16777216.\]
The polynomial $\cR_1(Z)$ has exactly 4 real zeroes with $Z\in(0,1)$, 
approximately located at the following values of
$z$:
\[-0.071519103755,\quad -0.114735617843,\quad -0.330240264422,\quad
-0.736842605000,\]
whereas
$\cR_2(Z)$ has exactly two zeroes in the same interval corresponding to the values of $z$
\[ z_4\approx -0.7389177458229170,\quad z_5\approx -0.2839588732787964. \]
This fact is in agreement with the
plots in Figure \ref{fig:stabrest}. 
To show that the two zeroes $z_4,z_5$ for $\Gamma=0$ give rise to a
bifurcation, it is enough to check that 
$\partial$Res$(\Gamma,z,D(z))/\partial \Gamma \neq0$ 
at the points $(0,z_j,D(z_j)),j=4,5.$ The computed values are
$\approx -0.626660386126$ at $z_4$ and $\approx 27.6667376231$ at $z_5$, far
away from zero.

In the case of the other four double zeroes, we obtain values of
$\partial$Res$(\Gamma,z,D(z))/\partial\Gamma$ equal to zero (with the expected
accuracy, see the beginning of the present section). Imposing the condition of
double zero for $M$, $\partial \hat{p}(M,\Gamma,z,D(z))/\partial M=0$, and
substituting it in $\hat{p}(M,\Gamma,z,D(z))$, we obtain, to low order, a double branch
of double zeroes in the $(\Gamma,z)$ variables. Using now Res$_3$, as we did
in Subsection \ref{sect:near00}, we obtain that the branch is simple. There is
no need to employ the Newton polygon; the Implicit Function Theorem is enough because the linear coefficients are nonzero. Hence, no bifurcation related to the zeroes of $\cR_1$ occurs.

We have thus proved the following result.

\begin{proposition} \label{prop:gammasmall}
In the passage from the restricted to the general problem for $z<0$ and a small mass
ratio $\Gamma$, excluding a neighbourhood of $(\Gamma,z)=(0,0),$
already studied in Proposition \ref{prop:near00}, changes in the stability
properties occur along lines of the $(\Gamma,z)$ plane. Three of these lines tend
to the values $-z_1,-z_2,-z_3$ when $\Gamma\to 0$. Additional changes occur
along two curves, with quadratic tangencies to the line $\Gamma=0$ at the two values
$z_4$ and $z_5$ where the characteristic polynomials of the curved $3$-body
problem and the restricted problem have zeroes in common. In all these cases a
Hamiltonian-Hopf bifurcation occurs.
\end{proposition}

\medskip

\subsection{\bf The main result}
We can summarize the conclusions obtained in the above propositions by saying
that they validate the numerical results near the relevant boundaries of the domain
$(\Gamma,z)$, i.e.\ near $\Gamma=0, \Gamma=\infty, z=0$. We have found all bifurcations produced by perturbation of the limit problems.
These properties also describe a general view on the problem of stability of tetrahedral orbits in
the curved 4-body problem in $\mathbb S^2$. The properties rigorously proved above are now summarized by the following result.

\begin{theorem}
We consider the tetrahedral solutions of the positively curved $4$-body problem on $\mathbb S^2$ with a fixed
body of mass $m_1$ located at the north pole and the other three bodies of equal mass $m$ located at the vertices
of an equilateral triangle orthogonal to the $z$-axis. Let be $\Gamma = \frac{m_1}{m}>0$. Then:
\begin{enumerate}
\item For $\Gamma \to \infty$ there are three functions, $z_4(1/\Gamma),
z_5(1/\Gamma)$, and $z_6(1/\Gamma)$, tending to $z_{4,0}, z_{5,0}$, and $z_{6,0}$,
respectively, at which Hamiltonian-Hopf bifurcations occur. The respective tetrahedral relative equilibria are
\begin{itemize}
\item spectrally stable of type ${\rm E}^6$ for $z \in (0,z_4(1/\Gamma)).$
\item unstable of type ${\rm E}^4{\rm CS}^1$ for $z\in(z_4(1/\Gamma),z_5(1/\Gamma))\cup(z_6(1/\Gamma),1)$, and of type ${\rm E}^2{\rm CS}^2$ for $z\in(z_5(1/\Gamma),z_6(1/\Gamma))$.
\end{itemize}
\item For any fixed value of the mass ratio $\Gamma$ in $(0,\infty)$, there is
a range of values of $z$ of the form $(0,\hat{z}(\Gamma)),\hat{z}(\Gamma)>0$,
such that in that range the orbit is spectrally stable. The value of 
$\hat{z}(\Gamma)$ tends to $z_{4,0}$ for $\Gamma\to\infty$ and $\hat{z}(\Gamma)
\to 0$ for $\Gamma\to 0$.
\item For any $z>0$ fixed when $\Gamma$ approaches 0 the orbit is unstable.
\item For any $z<0$ fixed and close to zero, let $\Gamma^*(z)$ be the value
corresponding to $\Omega^2 =0$ and $\gamma=\Gamma^*-\Gamma$. Then, when
$\gamma\to 0^+$, the corresponding tetrahedral relative equilibria are unstable.
\item Around the point $\Gamma=0,z=0$ there are six sectors, $\sigma_1, \sigma_2,
 \sigma_3, \sigma_4, \sigma_5, \sigma_6$, ordered counterclockwise, in which the
type of the orbits are: ${\rm E}^6,\,{\rm E}^5{\rm H}^1,\,{\rm E}^3{\rm H}^1
{\rm CS}^1$, no solutions, ${\rm E}^6$, and ${\rm E}^4{\rm CS}^1$, respectively.
The dominant terms in the boundaries of the sectors are of the form $\Gamma=0,\,
\Gamma=-c_1z^3,\,z=-\sqrt{3}\Gamma/8-c_2\Gamma^2,$ $z=-\sqrt{3}\Gamma/8-c_3
\Gamma^3,\,z=0, \,z=c_4\Gamma^{2/3}$, and $\Gamma=0$, the first four with $z<0$
and last two with $z>0$. All the coefficients  $c_i$ are positive.
\item For small $\Gamma$ and $z<0$, there exist five curves, $\psi_1, \psi_2,
 \psi_3, \psi_4$, and $\psi_5$, at which stability changes of Hamiltonian-Hopf
type occur. The first three are transversal to the line $\Gamma=0$, whereas the other two
are tangent. These curves reach $\Gamma=0$ for the following values of $z$:
$$z_1 \approx -0.829985,\ \ z_2\approx  -0.731860,\ \ z_3\approx  -0.370248,$$
$$z_4\approx -0.738918,\ \ z_5\approx -0.283959.$$
For small\ $\Gamma$, in particular, the tetrahedral relative equilibria are
spectrally stable for $z$ between: $\psi_1$ and the lower branch of $\psi_4$;
the upper branch of $\psi_4$ and $\psi_2$; $\psi_3$ and the lower branch of 
$\psi_5$; and between the upper branch of $\psi_5$ and the curve
bounding $\sigma_1$ in item {\rm (5)}, above.
\end{enumerate}
\end{theorem}

\section{Conclusions and outlook}

In this last section we will draw some final conclusions about the stability of tetrahedral relative equilibria and propose three problems that, in order to be solved, would require certain refinements of the methods we applied here.

\begin{figure}[ht]
\begin{center}
\epsfig{file=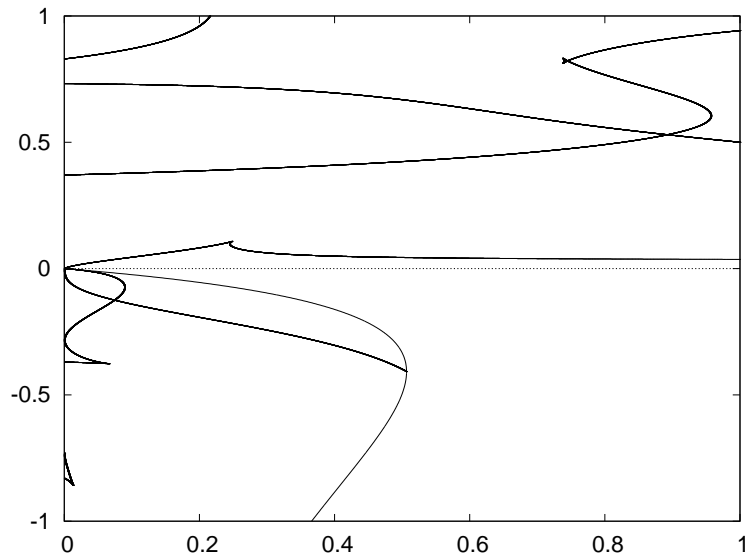,width=11cm}
\end{center}
\vspace{-2mm}
\caption{Bifurcation diagram for the relative equilibrium solutions. Variables
displayed: $(\Gamma/(1+\Gamma),z)$. With the only exception of the boundary of
the admissible domain $\Omega^2=0$ for $z<0$ and the line going from $(0,0)$ to
the tip of the $\Omega^2=0$ line, which corresponds to E-H bifurcation, all the
other lines correspond to Hamiltonian-Hopf bifurcations.}
\label{fig:bifur}
\end{figure}

We can now summarize the stability results we obtained in this paper about the relative equilibria of the tetrahedral 4-body problem in $\S^2$ by displaying the full
bifurcation diagram. To complete the above analysis of the limit cases, we
present the diagram computed from the resolvent Res$(\Gamma,
z,D(z))$ and from the conditions $\hat{p}(0,\Gamma,z,D(z))=0$. In both cases,
given a value of $z$, we obtain a polynomial equation for $\Gamma$. We computed the zeroes numerically and discarded the ones that do not give rise to any bifurcation. We checked the facts that occur here by looking at the derivatives with
respect to $M$ and $z$ at the solutions found. Figure \ref{fig:bifur} depicts
the results. As horizontal variable we used $\Gamma/(1+\Gamma)$ in order to display the full range of $\Gamma\in[0,\infty]$.

\bigskip

A possible continuation of the present work is the study of the linear stability for pyramidal solutions, i.e.\ orbits of the positively curved $n$-body problem, for $n>4$, with a fixed body of mass $m_1$ located at the north pole and the other $n-1$ bodies of equal mass $m$ lying at the vertices of a rotating regular polygon, orthogonal to the $z$-axis. But the analytic methods pursued here have limits. It seems that the symbolic computations and the related analysis would become insurmountable for $n$ larger than 7 or 8. Even a purely numerical study must be done very carefully. Another interesting problem is to analyze the linear stability of tetrahedral orbits in $\S^3$.
Finally, the stability of tetrahedral orbits in $\S^2$ when the $z$-coordinate of the three equal masses is not constant, but varies periodically in time, would also be a problem worth approaching.

\section*{Acknowledgements}
This research has been supported in part by Grants  
MTM2006-05849/Consolider and  MTM2010-16425 from Spain 
(Regina Mart\'{\i}nez and Carles Sim\'o), Conacyt Grant 128790 from M\'exico 
(Ernesto P\'erez-Chavela), and NSERC Discovery Grant 122045 from Canada (Florin Diacu).  The authors also acknowledge the computing facilities of the Dynamical Systems Group at the Universitat de Barcelona, which have been largely used in the numerical experiments presented in this paper.

\bigskip



\begin{thebibliography}{99}


\bibitem{ber} J.~Bertrand, Th\'eor\`eme relatif au mouvement d'un point attir\'e vers
un centre fixe, {C.\ R.\ Acad.\ Sci.}~{\bf 77} (1873), 849-853.

\bibitem{bol} W.~Bolyai and J.~Bolyai, {\it Geometrische Untersuchungen}, Hrsg. 
P.~St\"ackel, Teubner, Leipzig-Berlin, 1913.

\bibitem{diacu1} F.~Diacu, Near-collision dynamics for particle systems with quasihomogeneous potentials, {\it J.\ Differential Equations} {\bf 128} (1996), 58-77.

\bibitem{diacu2} F.\ Diacu, On the singularities of the curved $n$-body problem, {\it Trans.\ Amer.\ Math.\ Soc.} {\bf 363}, 4 (2011), 2249-2264.

\bibitem{diacu3} F.\ Diacu, Polygonal homographic orbits of the curved 3-body problem, {\it Trans.\ Amer.\ Math.\ Soc.} {\bf 364} (2012), 2783-2802.

\bibitem{diacu4} F.\ Diacu, Relative equilibria in the 3-dimensional curved $n$-body problem, arXiv:1108.1229.

\bibitem{diacu5} F.\ Diacu, {\it Relative equilibria of the curved $N$-body problem},
Atlantis Monographs in Dynamical Systems, Atlantis Press, Amsterdam, 2012 (to appear).

\bibitem{diacu6} F.\ Diacu, {\it The non-existence of the centre-of-mass and the linear-momentum integrals in the curved $N$-body problem}, arXiv:1202.4739.

\bibitem{diacufupersan} F.~Diacu, T.~Fujiwara, E.~P\'erez Chavela, and M.~Santoprete,
Saari's homographic conjecture of the 3-body problem, {\it Trans.~Amer.~Math.~Soc.} {\bf 360}, 12 (2008), 6447-6473.

\bibitem{diacuper} F.~Diacu and E.~P\'erez Chavela, Homographic solutions of the curved $3$-body problem, {\it J.\ Differential Equations} {\bf 250} (2011), 340-366.

\bibitem{diacuperrey} F.~Diacu, E.~P\'erez Chavela, and J.G.\ Reyes Victoria, An intrinsic approach in the curved $n$-body problem. The negative curvature case, \emph{J.\ Differential Equations} {\bf 252}, 8 (2012), 4529-4562.

\bibitem{diacupersan1} F.~Diacu, E.~P\'erez Chavela, and M.~Santoprete, Saari's conjecture for the collinear $n$-body problem, {\it Trans.~Amer.~Math.~Soc.} {\bf 357}, 10 (2005), 4215-4223. 

\bibitem{diacupersan2} F.~Diacu, E.~P\'erez Chavela, and M.~Santoprete, The
$n$-body problem in spaces of constant curvature. Part I: Relative equilibria,
{\it J.\ Nonlinear Sci.} {\bf 22}, 2 (2012), 247-266, DOI: 10.1007/s00332-011-9116-z.

\bibitem{diacupersan3} F.~Diacu, E.~P\'erez Chavela, and M.~Santoprete, The
$n$-body problem in spaces of constant curvature. Part II: Singularities,
{\it J.\ Nonlinear Sci.} {\bf 22}, 2 (2012), 267-275, DOI: 10.1007/s00332-011-9117-y.

\bibitem{ein}  A.~Einstein, L.~Infeld, and B.~Hoffmann, The gravitational equations
and the problem of motion, {\it Ann.\ of Math.} {\bf 39}, 1 (1938), 65-100.

\bibitem{foc} V.~A.~Fock, Sur le mouvement des masses finie d'apr\`es
la th\'eorie de gravitation einsteinienne, {\it J.~Phys.~Acad.~Sci.~USSR}
{\bf 1} (1939), 81-116.

\bibitem{kil1} W.~Killing, Die Rechnung in den nichteuklidischen Raumformen,
{\it J. Reine Angew. Math.} {\bf 89} (1880), 265-287.

\bibitem{kil2} W.~Killing, Die Mechanik in den nichteuklidischen Raumformen,
{\it J. Reine Angew. Math.} {\bf 98} (1885), 1-48.

\bibitem{kil3} W.~Killing, {\it Die Nicht-Eukildischen Raumformen in Analytischer
Behandlung}, Teubner, Leipzig, 1885.

\bibitem{koz} V.~V.~Kozlov and A.~O.~Harin, Kepler's problem in constant
curvature spaces, {\it Celestial Mech.~Dynam.~Astronom} {\bf 54} (1992), 393-399.

\bibitem{civ1} T.~Levi-Civita, The relativistic problem of several bodies, {\it Amer.\ J.\ Math.} {\bf 59}, 1 (1937), 9-22.

\bibitem{civ2} T.~Levi-Civita, {\it Le probl\`eme des n corps en relativit\'e g\'en\'erale}, Gauthier-Villars, Paris, 1950; or the English translation: {\it The $n$-body problem in general relativity}, D.~Reidel, Dordrecht, 1964.

\bibitem{lie1} H.~Liebmann, Die Kegelschnitte und die Planetenbewegung im 
nichteuklidischen Raum, {\it Berichte K\"onigl.~S\"achsischen Gesell. Wiss., Math.~Phys.~Klasse} {\bf 54} (1902), 393-423.

\bibitem{lie2} H.~Liebmann, \"Uber die Zentralbewegung in der nichteuklidische
Geometrie, {\it Berichte K\"onigl.~S\"achsischen Gesell. Wiss., Math.~Phys.~Klasse} {\bf 55} (1903), 146-153.

\bibitem{lie3} H.~Liebmann, {\it Nichteuklidische Geometrie}, G.~J.~G\"oschen, Leipzig, 1905; 2nd ed.~1912; 3rd ed.~Walter de Gruyter, Berlin, Leipzig, 1923.

\bibitem{lip1} R.~Lipschitz, Untersuchung eines Problems der Variationrechnung, in welchem das Problem der Mechanik enthalten ist, {\it J. Reine Angew. Math.} {\bf 74} (1872), 116-149.

\bibitem{lip2} R.~Lipschitz, Extension of the planet-problem to a space of $n$ dimensions and constant integral curvature, {\it Quart.~J.~Pure Appl.~Math.} {\bf 12} (1873), 349-370.

\bibitem{lob} N.~I.~Lobachevsky, The new foundations of geometry with full theory of parallels [in Russian], 1835-1838, In Collected Works, V. 2, GITTL, Moscow, 1949, p. 159.

\bibitem{masimo} R.\ Mart\'inez and C.\ Sim\'o, On the stability of the Lagrangian homographic solutions in a curved three-body problem on $\mathbb S^2$, {\it Discrete Contin. Dyn. Syst. Ser. A} (to appear).

\bibitem{perezrey} E.~P\'erez Chavela and J.G.~Reyes Victoria, An intrinsic approach in the curved $n$-body problem. The positive curvature case, {\it Trans.\ Amer.\ Math.\ Soc.} {\bf 364}, 7 (2012), 3805-3827.

\bibitem{sch1} E.~Schering, Die Schwerkraft im Gaussischen R\"aume, 
{\it Nachr. K\"onigl. Gesell. Wiss. G\"ottingen} 13 July, {\bf 15} (1870), 311-321.

\bibitem{sch2} E.~Schering, Die Schwerkraft in mehrfach ausgedehnten  Gaussischen und Riemmanschen R\"aumen, {\it Nachr. K\"onigl. Gesell. Wiss. G\"ottingen} 26 Febr., {\bf 6} (1873), 149-159.

\bibitem{ser} P.J.~Serret, {\it Th\'eorie nouvelle g\'eom\'etrique et m\'ecanique des lignes a double courbure}, Mallet-Bachelier, Paris, 1860.

\bibitem{shc1} A.V.~Shchepetilov, Comment on ``Central potentials on spaces of
constant curvature: The Kepler problem on the two-dimensional sphere ${\mathbb S}^2$ and the hyperbolic plane ${\mathbb H}^2$,'' [J.~Math.~Phys.~{\bf 46} (2005),  052702], \emph{J.~Math.~Phys.}~{\bf 46} (2005), 114101.

\bibitem{shc2} A.V.~Shchepetilov, Reduction of the two-body problem with central
interaction on simply connected spaces of constant sectional curvature, {\it J.~Phys.~A: Math.~Gen.} {\bf 31} (1998), 6279-6291.

\bibitem{shc3} A.V.~Shchepetilov,  Nonintegrability of the two-body problem in constant curvature spaces, {\it J.\ Phys. A: Math.\ Gen.} V.\ 39 (2006), 5787-5806; corrected version at math.DS/0601382.

\bibitem{shc4} A.V.~Shchepetilov, {\it Calculus and mechanics on two-point homogeneous Riemannian spaces}, Lecture notes in physics, vol.~707, Springer
Verlag, 2006.

\bibitem{meer} J.C.\ van der Meer, {\it The Hamiltonian Hopf Bifurcation}, Springer Verlag, 1985.

\end{thebibliography}
\end{document}